\documentclass [11pt,reqno]{amsart}
\usepackage {amsmath,amssymb,verbatim,geometry,color,esint}
\usepackage[all]{xy}
\usepackage{esint}
\usepackage{mathrsfs}
\usepackage[pdftex,hyperindex]{hyperref}
\usepackage[pdftex]{graphicx}
\usepackage{tikz}
\usetikzlibrary{math}
\usepackage{tikz-cd}
\usepackage{float}
\usepackage[justification=centering]{caption}

 \newcommand{\C}{\mathbb{C}}

\newcommand{\Gm}{\mathbb{G}_\mathrm{m}}

\renewcommand{\P}{\mathbb{P}}
 \newcommand{\Q}{\mathbb{Q}}
 \newcommand{\R}{\mathbb{R}}
 \newcommand{\Z}{\mathbb{Z}}

\newcommand{\refe}{\mathrm{ref}}

\newcommand{\cM}{\mathcal{M}}

\newcommand{\cO}{\mathcal{O}}
\newcommand{\cP}{\mathcal{P}}
\newcommand{\cQ}{\mathcal{Q}}
\newcommand{\cR}{\mathcal{R}}

\newcommand{\cX}{\mathcal{X}}
\newcommand{\cY}{\mathcal{Y}}

\renewcommand{\a}{\alpha}
\renewcommand{\b}{\beta}

\newcommand{\e}{\varepsilon}
\newcommand{\f}{\varphi}
\newcommand{\unipar}{t}

\newcommand{\om}{\omega}

\newcommand{\p}{\psi}

\newcommand{\eg}{{\rm e.g.\ }} 
\newcommand{\ie}{{\rm i.e.\ }} 
\newcommand{\loccit}{\textit{loc.\,cit.\ }}

\newcommand{\an}{\mathrm{an}}

\newcommand{\redu}{\mathrm{red}}

\DeclareMathOperator{\FS}{FS}

\DeclareMathOperator{\Hom}{Hom}

\DeclareMathOperator{\MA}{MA}
\DeclareMathOperator{\NA}{NA}

\DeclareMathOperator{\Spec}{Spec}

\DeclareMathOperator{\trop}{trop}

\DeclareMathOperator{\id}{id}

\DeclareMathOperator{\sym}{sym}

\DeclareMathOperator{\Star}{Star}

\DeclareMathOperator{\Sing}{Sing}
\DeclareMathOperator{\Sk}{Sk}

\DeclareMathOperator{\Log}{Log}

\DeclareMathOperator{\Hnot}{H^0}

\renewcommand{\div}{\mathrm{div}}

\newcommand{\val}{\mathrm{val}}

\newcommand{\Aspace}{A}
\newcommand{\Bspace}{B}

\newcommand{\cro}[1]{[\![#1]\!]}
\newcommand{\lau}[1]{(\!(#1)\!)}
\newcommand{\simto}{\overset\sim\to}

\newcommand{\ch}{\tiny \vee}

\newcommand{\MANA}{\MA_{\NA}}
\newcommand{\MAR}{\MA_{\R}}

\newcommand{\ti}[1]{\widetilde{#1}}

\numberwithin{equation}{section}       % Number formulas within sections

\newtheorem{prop} {Proposition} [section]
\newtheorem{thm}[prop] {Theorem} 
\newtheorem{defi}[prop] {Definition}
\newtheorem{lem}[prop] {Lemma}
\newtheorem{cor}[prop]{Corollary}
\newtheorem{prop-def}[prop]{Proposition-Definition}

\newtheorem*{thmA}{Theorem A} 
 
\newtheorem*{thmB}{Theorem B}

\newtheorem*{corC}{Corollary C}

\newtheorem{exam}[prop]{Example}
\newtheorem{rmk}[prop]{Remark}

\theoremstyle{remark}

\newcommand{\corr}[1]{{\color{black}{#1}}}

\title[Monge--Amp\`ere equations]{Tropical and non-Archimedean Monge--Amp\`ere equations for a class of Calabi--Yau hypersurfaces}
\date{\today}

\author{Jakob Hultgren \and Mattias Jonsson \and Enrica Mazzon \and Nicholas McCleerey}

\address{Dept of Mathematics and Mathematical Statistics\\
  Ume{\aa} University\\ 
  901 87 Ume{\aa}, Sweden \\and Dept of Mathematics\\University of Maryland\\College Park, MD 20742-4015\\ USA}

\email{jakob.hultgren@umu.se}
  
\address{Dept of Mathematics\\
  University of Michigan\\
  Ann Arbor, MI 48109-1043\\ USA}

\email{mattiasj@umich.edu}

\address{Fakult\"at für Mathematik\\
Universit\"at Regensburg\\
93040 Regensburg\\
Germany}

\email{e.mazzon15@alumni.imperial.ac.uk}

\address{Dept of Mathematics\\
Purdue University\\
West Lafayette, Indiana 47907-2067\\
USA}  

\email{nmccleer@purdue.edu}

\begin{document}

\begin{abstract}
    For a large class of maximally degenerate families of Calabi--Yau hypersurfaces of complex projective space, we study non-Archimedean and tropical Monge--Amp{\`e}re equations, taking place on the associated Berkovich space,  
    and the essential skeleton therein, respectively. For a symmetric measure on the skeleton, we prove that the tropical equation admits a unique solution, up to an additive constant. Moreover, the solution to the non-Archimedean equation can be derived from the tropical solution, and is the restriction of a continuous semipositive toric metric on projective space. Together with the work of Yang~Li, this implies the weak metric SYZ conjecture on the existence of special Lagrangian fibrations in our setting.
 \end{abstract}

\maketitle

\setcounter{tocdepth}{1}
\tableofcontents
% 
%
%%%%%%%%%%%%%%%%%%%%%%%%%%%%%%%%%%%%%%%%%%%%%%%%%%%%%%%%%%%%%%%%%%%
%
%
\section*{Introduction}
Let $f(z)\in\C[z_0,\dots,z_{d+1}]$ be a generic homogeneous polynomial of degree $d+2$, where $d\ge1$. Then 
\begin{equation}\label{eq:Family}
  X:=\{z_0z_1\dots z_{d+1}+\unipar f(z)=0\}\subset\P^{d+1}\times\C^*\tag{$\star$}
\end{equation}
defines a maximally degenerate 1-parameter family of complex Calabi--Yau manifolds, polarized by $L:=\cO(d+2)|_X$. By Yau's theorem~\cite{Yau}, we can equip each $X_\unipar$ with a Ricci flat metric in the Chern class of $L|_{X_\unipar}$. 
The structure of $X_\unipar$ as $\unipar \to 0$ is described by two fundamental conjectures, namely the \emph{SYZ conjecture}~\cite{SYZ} and the~\emph{Kontsevich--Soibelman conjecture}~\cite{KS06}. These two conjectures have recently been related to a conjecture about solutions to the \emph{non-Archimedean Monge-Amp{\`e}re equation} \cite{LiSYZ}.
In this paper we address the latter conjecture and prove a weak version of the SYZ conjecture in the setting above.

To explain all this, first note that $X$ defines a smooth projective variety over the non-Archimedean field $K:=\C\lau{\unipar}$ of complex Laurent series. Its Berkovich analytification $X^\an$ has a canonical closed subset $\Sk(X)\subset X^\an$, the \emph{essential skeleton},~\cite{KS06,MN15}, which in this case can be identified with the boundary of a $(d+1)$-dimensional simplex. The skeleton has a canonical piecewise integral affine structure, and in particular a canonical Lebesgue measure.

The Kontsevich--Soibelman conjecture states that, as $\unipar\to0$, $X_\unipar$ converges (after rescaling) in the Gromov--Hausdorff sense to a metric space whose underlying topological space is $\Sk(X)$, and whose metric is determined by the solution to---roughly speaking---a real Monge--Amp\`ere equation on the skeleton, with right hand side given by the Lebesgue measure on $\Sk(X)$. Making sense of this equation is not obvious, but something that we address satisfactorily in Theorem~B below in our setting.

As an alternative, one can look at the non-Archimedean Monge--Amp\`ere equation. To any continuous semipositive metric $\|\cdot\|$ on $L^\an$ is associated a \emph{Chambert--Loir measure} $c_1(L,\|\cdot\|)^d$, a positive Radon measure on $X^\an$ of mass $(d+2)^{d+1}$~\cite{CL06,Gub07}. By the main results in~\cite{nama,YZ17}, any positive Radon measure $\nu$ on $\Sk(X)$ of this mass is the Chambert--Loir measure of a continuous semipositive metric, unique up to scaling.

When $\nu$ equals Lebesgue measure on the skeleton, it is expected that the solution to the non-Archimedean Monge--Amp\`ere equation can be used to define the metric in the Kontsevich--Soibelman conjecture, as explored by Yang Li in his groundbreaking work~\cite{LiSYZ} (see also~\cite{LiSurvey}).
Unfortunately, the proof in~\cite{nama} is variational in nature, and does not give any information beyond continuity. 

Our first main result gives a much more precise description of the solution in terms of convex functions or, put differently, toric metrics. 
\begin{thmA}
  If $\nu$ is a symmetric positive measure on $\Sk(X)$ of mass $(d+2)^{d+1}$, then any solution to $c_1(L,\|\cdot\|)^d=\nu$ is the restriction of a symmetric toric metric on $\cO_{\P^{d+1}}(d+2)^\an$.
\end{thmA}
Let us be a bit more precise. In Theorem~A we assume that the polynomial $f(z)$ used to define $X$ is \emph{admissible} in the following sense: for any intersection $Z$ of coordinate hyperplanes $z_j=0$ in $\P^{d+1}$, $f$ does not vanish identically on $Z$ and $V(f|_Z)$ is smooth, see~\S\ref{sec:CY}. A general polynomial is admissible.

The symmetric group $S_{d+2}$ acts on projective space and its analytification by permuting the coordinates $z_i$. This action preserves $\Sk(X)$, but not necessarily $X^\an$. We say that a measure $\nu$ on $\Sk(X)$ is \emph{symmetric} if it is invariant under the action. For example, Lebesgue measure is symmetric.

A particular example of an admissible polynomial is the Fermat polynomial $f(z)=\sum_0^{d+1}z_i^{d+2}$. The resulting \emph{Fermat family} is the central object in \cite{LiFermat}. For this family, Theorem~A was obtained independently by Pille-Schneider~\cite{PS22} in the special case when $\nu$ is the Lebesgue measure, by using the results from~\cite{LiFermat}.

\smallskip
To prove Theorem~A we study the real Monge--Amp\`ere equation on the skeleton $\Sk(X)$, as alluded to above. In doing so we exploit the structure of $X\subset \P^{d+1}$, as in \cite{LiFermat}. Namely, we view $\P^{d+1}$ as a toric variety with character lattice $M$ and co-character lattice $N$. Let $\Delta\subset M_\R$ be the polytope for the anticanonical bundle $\cO(d+2)$ on $\P^{d+1}$. There is a bijection between continuous semipositive toric metrics on $\cO_{\P^{d+1}}(d+2)^\an$ and convex functions $\psi\colon N_\R\to\R$ whose Legendre transforms are continuous convex functions on $\Delta$.

Both $\Delta$ and its polar $\Delta^\vee\subset N_\R$ are $(d+1)$-dimensional simplices. It turns out that the boundary $\Bspace:=\partial\Delta^\vee$ can be identified with the essential skeleton of $X$; we therefore work on $\Bspace$ rather than $\Sk(X)$. Let $\cQ\subset C^0(\Bspace)$ be the set of restrictions $\psi|_\Bspace$, with $\psi$ as above, and $\cQ_{\sym}\subset\cQ$ the subset of $S_{d+2}$-invariant functions. 

Each $d$-dimensional face $\tau_i$ of $\Bspace$ comes with an integral affine structure, and the restriction of any $\psi\in\cQ$ to $\tau_i$ is a convex function. This allows us to define the real Monge--Amp\`ere measure $\MAR(\psi|_{\tau_i^\circ})$ on the interior $\tau_i^\circ$ of $\tau_i$. We show that this Monge--Amp\`ere operator extends naturally to all of $\Bspace$, at least for symmetric functions. Let $\cM_{\sym}$ denote the space of positive, symmetric measures on $\Bspace$ of mass $(d+2)^{d+1}/d!$. 
\begin{thmB}
  There exists a unique continuous map $\cQ_{\sym}\ni \psi\mapsto \nu_\psi\in\cM_{\sym}$ such that
  \begin{equation}
    \nu_\psi|_{\tau_i^\circ}=\MAR(\psi|_{\tau_i^\circ})\tag{$\dagger$}
  \end{equation}
  for all $\psi\in\cQ_{\sym}$ and all $i$. Moreover, this map induces a homeomorphism $\cQ_{\sym}/\R\to\cM_{\sym}$.
\end{thmB}
The space $\Bspace$ is an integral tropical manifold in the sense of~\cite{Gro13}, and Theorem~B can be seen as solving a \emph{tropical Monge--Amp\`ere equation}; slightly more precisely we can define a natural integral affine structure on a subset $\Bspace_0\subset \Bspace$, with $\Bspace\setminus \Bspace_0$ of codimension 2. Any $\psi\in\cQ_{\sym}$ can then be viewed as a convex metric on a certain affine $\R$-bundle over $\Bspace_0$, in the sense of~\cite{HO19}, 
with real Monge--Amp\`ere measure $\nu_{\psi}|_{\Bspace_0}$; see~\S\ref{sec:AffRb} and~\S\ref{rmk:MAMetric} for details.
While the real Monge--Amp\`ere measure of this convex metric is only defined on $B_0$, Theorem~B gives a way of extending this operator over the singular set $B\setminus B_0$. 

After the first draft of this paper appeared, it was pointed out to us by Rolf Andreasson that the main result of~\cite{Caf} directly gives 
a regularity result for solutions $\psi$ to $\nu_\psi = \mu$ which implies they define smooth Hessian metrics
over $B_0$ when $\mu$ is the Lebesgue measure on $A$. See \cite[Theorem~3, Lemma~16 and Lemma~17]{AH} for details and an extension to other symmetric polytopes.

\smallskip
Combining Theorem~B and its proof with the work of Li~\cite{LiSYZ} we obtain a weak version of the SYZ conjecture in our setting. The SYZ conjecture predicts that $X_\unipar$ admits a special Lagrangian fibration for small $\unipar$.\footnote{Ruddat and Siebert proved that~$X_0$ itself admits a special Lagrangian fibration, see~\cite{RS20}.} 
\begin{corC}
    Given $\delta>0$, for all sufficiently small $\unipar$ there exist a special Lagrangian torus fibration on an open subset of $X_\unipar$ of normalized Calabi--Yau volume at least $1-\delta$;
\end{corC}
This is stronger than the main result of~\cite{LiFermat}, as the analysis in~\loccit is restricted to the Fermat family, where $f(z)=\sum_{j=0}^{d+1}z_j^{d+2}$, and to subsequences $X_{\unipar_n}$, with $\unipar_n\to0$.

In~\cite{LiSYZ}, Li gave an argument
reducing Corollary~C (as well as the corresponding statement for more general families) to
a certain conjectural \emph{comparison property} of the solution to the non-Archimedean Monge--Amp\`ere equation. In fact, our proof of Corollary~C follows~\cite{LiSYZ}, using a weaker version of the comparison property that we derive from Theorem~A and its proof.

\smallskip

Li also proved a weak version of the Kontsevich--Soibelman conjecture for the Fermat family in~\cite{LiFermat}: any subsequential Gromov--Hausdorff limit of $X_\unipar$ as $\unipar\to0$ contains a dense subset locally isometric to the regular part of a Monge--Amp{\`e}re metric on $B_0$. 
The injectivity in Theorem~B implies that the dense set in these subsequential Gromov--Hausdorff limits is uniquely determined up to local isometry, as is also obtained in~\cite{PS22}.

The solution $\psi\in\cQ_{\sym}$ to the equation $\nu_\psi=\nu$, where $\nu$ is Lebesgue measure on $\Bspace$, can be used to state a precise version of the Kontsevich--Soibelman conjecture in this setting. Namely, if we knew that (the metric associated to) $\psi$ is smooth and strictly convex on $\Bspace_0$, then its Hessian would give $\Bspace_0$ the structure of a metric space, whose completion should be homeomorphic to $\Bspace$, and equal to the Gromov--Hausdorff limit of $X_\unipar$ as $\unipar\to0$. It seems plausible that having a well-posed global (tropical) Monge--Amp\`ere equation may allow us to improve the local regularity results~\cite{Figalli,Moo15,Moo21,MR22}, which themselves are not sufficient, at least in dimension $d\ge3$.

\smallskip
See also~\cite{CJL,Got22,GO22,GTZ13,GTZ16,GW00,MMRZ22,Oda20,OO21,RZ22} for related, but slightly different, approaches to the SYZ and Kontsevich--Soibelman conjectures. In particular, a version of the Kontsevich--Soibelman conjecture is known in dimension 2~\cite{GW00,OO21}.

\subsection*{Strategy}
We now describe the main ideas behind Theorem~B. While there are satisfactory results for the Monge--Amp\`ere equation on Hessian manifolds~\cite{ChengYau,Del89,HO19,GuedjTo}, extending these to general integral tropical manifolds  seems challenging. Instead, our approach heavily uses the large symmetry group of $\Bspace$; 
this allows us to adapt the variational approach in~\cite{BB13,BBGZ,nama} for solving real, complex, and non-Archimedean Monge--Amp\`ere equations, respectively.

More precisely, if $\Aspace:=\partial\Delta\subset M_\R$, then the canonical pairing of 
$M_\R$ and $N_\R$ induces a \emph{cost function} on $\Aspace\times \Bspace$, in the sense of optimal transport. From this, one defines the $c$-transform (generalizing the usual Legendre transform), which can be used to recover $\cQ$ as the class of $c$-convex functions, and to define a notion of $c$-subgradients.

While the $c$-transform and $c$-subgradient express some pathological behavior in general, for symmetric functions, they reduce to the usual Legendre transform and subgradient when viewed in coordinate charts for the integral affine structure. 
For any $\psi\in\mathcal{Q}_{\sym}$, we may then define $\nu_\psi$ %associated to $\psi\in\cQ_{\sym}$
as the pushforward of Lebesgue measure on $\Aspace$ under the $c$-subgradient map of $\psi^c$, the $c$-transform of $\psi$.

Solving $\nu_\psi=\nu$, for a given $\nu \in \mathcal{M}_{\sym}$, can now be reformulated as minimizing a certain functional $F=F_\nu$ on $\cQ_{\sym}$; as
in~\cite{BB13,BBGZ,nama}
the crucial fact that the minimizer is a solution amounts to a differentiability property for $F$, which we can prove in the symmetric case (and, surprisingly, fails in the non-symmetric case, see Example~\ref{exam:non_diff}).
\smallskip

We now outline how to deduce Theorem~A from Theorem~B. For this, we need to explain the relation between $\Sk(X)$ and $\Bspace$. 

The variety $X$ admits a natural model $\cX$ over the valuation ring $\C\cro{\unipar}$, given by the same equation as above in \eqref{eq:Family}. Its special fiber $\cX_0$ is the union of the coordinate  hyperplanes in $\P^{d+1}_\C$, and the associated dual complex can be identified with $\Sk(X)$.
There are $d+2$ closed points $\xi_i\in\cX_0$ where $d+1$ distinct hyperplanes meet, and the preimage of $\xi_i$ under the specialization map $X^\an\to\cX_0$ 
is an open subset $U_i\subset X^\an$, whose intersection with $\Sk(X)$ is the relative interior $\tilde\tau_i^\circ$ of a $d$-dimensional simplex $\tilde\tau_i$; in fact, we have $\Sk(X)=\bigcup_i\tilde\tau_i$. We have a natural retraction $U_i\to\tilde\tau_i^\circ$, and this retraction is an affinoid torus fibration.

Let $T\subset\P^{d+1}$ be the torus. There is a canonical tropicalization map $\trop\colon T^\an\to N_\R$. One can show that $\Sk(X)\subset T^\an$, and that the tropicalization map restricts to a homeomorphism of $\Sk(X)$ onto $\Bspace$, sending $\tilde\tau_i$ onto $\tau_i$ for each $i$.
 On $U_i\subset X^\an$, the tropicalization map is also invariant under the retraction to $\tilde\tau_i^\circ$, and the restriction $\trop\colon U_i\to\tau_i^\circ$ is an affinoid torus fibration. 

Now consider the case of a symmetric measure $\nu$ on $\Sk(X)\simeq\Bspace$ that is sufficiently smooth, say equivalent to Lebesgue measure; the general case in Theorem~A can be treated by approximation. Pick $\psi\in\cQ_{\sym}$ with $\nu_\psi=\nu$. We can extend $\psi$ to a convex function on $N_\R$ whose Legendre transform is a symmetric continuous convex function on $\Delta$. As already mentioned, this induces a symmetric continuous semipositive toric metric on $\cO(d+2)^\an$, over $\P^{d+1,\an}$, and by restriction a continuous semipositive metric $\|\cdot\|$ on $L^\an$.

By construction, the restriction of $\|\cdot\|$ to $U_i$ can be viewed as the pullback of the convex function $\psi$ on $\tau_i^\circ$. Combining~($\dagger$) with a theorem of Vilsmeier in~\cite{Vilsmeier}, it follows that the Chambert-Loir measure $c_1(L,\|\cdot\|)^d$ agrees with the measure $\nu$ on an open subset of $\Sk(X)\simeq B$, and hence everywhere, as this open set carries all the mass of $\nu$.

\smallskip

Corollary~C relies on Theorem~A and the ideas of~\cite{LiSYZ}. Namely, while the model $\cX$ above is not semistable snc, Theorem~A implies that we still have the comparison property for the non-Archimedean and real Monge--Amp\`ere operators in the sense of~\cite[Definition~3.11]{LiSYZ}. The arguments in \textit{loc.\ cit.}\ then go through essentially unchanged; see~\S\ref{sec:SYZ} for details.
\smallskip

The variational principle we developed in Theorem B has been applied in some more general contexts after the first draft of this paper appeared. 
In particular, in~\cite{LiFano} it has been used to prove the SYZ for families of hypersurfaces in some toric Fano manifolds; this partially extends our approach to the non-symmetric setting, imposing however a
condition on the vertices of $\Delta$ and $\Delta^\vee$, which seems unfortunately rather restrictive.
In \cite{AH}, using to a larger extent the connections to optimal transport, Andreasson and Hultgren provide a necessary and sufficient condition for the solvability of the tropical Monge--Amp{\`e}re equation on a reflexive polytope, which implies the SYZ conjecture for the corresponding family of Calabi--Yau hypersurfaces.

\subsection*{Structure} The paper is organized as follows: after a discussion of the toric setup and the structure of $\Bspace$ as a tropical manifold, we introduce in~\S\ref{sec:cconvex} the class of c-convex functions, and show their basic properties. In~\S\ref{sec:sym}, we define the Monge--Amp\`ere operator on the subclass of symmetric c-convex functions, and in~\S\ref{sec:tropMA} we solve the tropical Monge--Amp\`ere equation, proving Theorem~B. The relation between c-convex functions and toric metrics on the Berkovich analytification on $\cO(d+2)$ is explored in~\S\ref{sec:NAmetrics}, whereas the restriction of the tropicalization map to $X^\an$ is studied in~\S\ref{sec:CY}. After that, combining all the ingredients, we prove Theorem~A in~\S\ref{sec:NAMA} and Corollary~C in~\S\ref{sec:SYZ}.
% 
%%%%%%%%%%%%%%%%%%%%%%%%%%%%%%%%%%%%%%%%%%%%%%%%%%%%%%%%%%%%%%%%%%%
%

\subsection*{Notation}
Given a variety $X$ over a non-Archimedean field $K$, we denote by $X^\an$ the Berkovich analytification of $X$, and by $X^\val\subset X^\an$ the subset of valuations on the function field of $X$ extending the valuation on $K$.
Given an abelian group $\Gamma$, we set $\Gamma_\R:=\Gamma\otimes_\Z\R$. If $g$ is a convex function on $\R^n$, then the subgradient $\partial g(x)$ at $x\in\R^n$ is the set of linear functions $\ell\in(\R^n)^*$ such that the function $g-\ell$ attains its minimum at $x$. The (real) Monge--Amp\`ere measure $\MAR(g)$ of  $g$ is taken in the sense of Alexandrov, \ie as the Lebesgue measure of the subgradient image, see \eg~\cite[\S2.1]{Figalli}.

\subsection*{Acknowledgement}
We thank Y.~Li and L.~Pille-Schneider for many clarifying remarks on their work, and T.~Darvas, V.~Guedj, H.~Ruddat, T.~D.~T\^{o}, Y.~Odaka, and V.~Tosatti for useful comments.
The first author was supported by the Knut and Alice Wallenberg Foundation grant 2018-0357, BSF grant 2020329, and NSF grant DMS-1906370. The second author was supported by NSF grants DMS-1900025 and DMS-2154380. The third author was supported by the collaborative research center SFB 1085 \emph{Higher Invariants - Interactions between Arithmetic Geometry and Global Analysis} funded by the Deutsche Forschungsgemeinschaft. 
% 
%
%%%%%%%%%%%%%%%%%%%%%%%%%%%%%%%%%%%%%%%%%%%%%%%%%%%%%%%%%%%%%%%%%%%
%
%
\section{Toric setup}
Fix an integer $d\ge1$.
 Our toric terminology largely follows~\cite{FultonToric}.
%
%%%%%%%%%%%%%%%%%%%%%%%%%%%%%%%%%%%%%%%%%%%%%%%%%%%%%%%%%%%%%%%%%%%
%
\subsection{Lattices and tori}
Consider the lattice $M':=\Z^{d+2}$ with basis
\[
  e_0=(1,0,\dots,0),\dots,e_{d+1}=(0,\dots,0,1).
\]
Let $T':=\Spec K[M']\simeq\Gm^{d+2}$ be the corresponding (split) torus. Each $m\in M'$
defines a character on $T'$. If we denote by $z_i$ the character associated to the basis element $e_i$, then the character associated to a general element
$m=(y_0,\dots,y_{d+1})\in M'$ is given by 
\[
  z^m:=z_0^{y_0}\cdot\ldots\cdot z_{d+1}^{y_{d+1}}
\]

Define a sublattice $M\subset M'$ by $M=\{y\in\Z^{d+2}\mid \sum_0^{d+1}y_i=0\}$. For any $i\in\{0,\dots,d+1\}$ the set $\{e_j-e_i\}_{j\ne i}$ forms a basis for $M$. Let $T:=\Spec K[M]\simeq\Gm^{d+1}$ be the associated torus. The inclusion $M\subset M'$ induces a morphism $T'\to T$, allowing us to view $T$ as a quotient of $T'$. The characters $z_i$ on $T'$ can be viewed as homogeneous coordinates on $T$.

Set $N':=\Hom(M',\Z)$ and $N:=\Hom(M,\Z)$. Then $N'\simeq\Z^{d+2}$ and 
\[
  N\simeq\Z^{d+2}/\Z(1,\dots,1).
\]

%
%%%%%%%%%%%%%%%%%%%%%%%%%%%%%%%%%%%%%%%%%%%%%%%%%%%%%%%%%%%%%%%%%%%
%
\subsection{Tropicalization}\label{subsec:trop map}
We use `additive' conventions for valuations and semivaluations. Thus $T^\an$ is the set of semivaluations $v\colon K[M]\to\R\cup\{+\infty\}$ restricting to the given valuation on $K$, and equipped with the topology of pointwise convergence. We have a tropicalization map
\[
  \trop\colon T^\an\to N_\R=\Hom(M,\R)
\]
characterized by 
\[
  \langle m,\trop(v)\rangle=-v(z^m)
\]
for all $m\in M$. This map is continuous and surjective. It admits a natural continuous one-sided inverse, which to $n\in N_\R$ associates the valuation $v_n\in T^\val\subset T^\an$ defined by
\[
  v_n\left(\sum_{m\in M}a_mz^m\right)=\min_m\{-v(a_m)-\langle m,n\rangle\};
\]
this is the minimal element in the fiber $\trop^{-1}(n)$, with respect to the natural partial ordering on $T^\an$.
% 
%%%%%%%%%%%%%%%%%%%%%%%%%%%%%%%%%%%%%%%%%%%%%%%%%%%%%%%%%%%%%%%%%%%
%
\subsection{Simplices and projective space}
Let $\Delta\subset M_\R$
%:=M\otimes_\Z\R$ 
be the convex hull of the elements
\[
  m_i:=(d+1)e_i-\sum_{j\ne i}e_j\in M,\quad i=0,\dots,d+1.
\]
Then $\Delta$ is a simplex, whose polar polytope\footnote{We use a different sign convention from~\cite{LiFermat}.}
\[
  \Delta^\vee:=\{n\in N_\R\mid\sup_{m\in\Delta}\langle m,n\rangle=\max_{0\le i\le d+1}\langle m_i,n\rangle\le 1\},
\]
is also a simplex, with vertices given by
\[
  n_0=(-1,0,\dots,0),\dots,n_{d+1}=(0,\dots,0,-1).
\]
The fan in $N_\R$ dual to $\Delta$ has rays generated by $n_i$, $0\le i\le d+1$, and defines a toric variety that we identify with $\P^{d+1}$. In fact, $\Delta$ is the moment polytope for the anticanonical bundle $\cO(d+2)$ on $\P^{d+1}$, and the unique effective torus invariant anticanonical divisor on $\P^{d+1}$ is given by $-K_{\P^{d+1}}=\sum_{i=0}^{d+1}D_i$, where $D_i$ is the prime divisor on $\P^{d+1}$ corresponding to $n_i$.

For later reference, we note that
\begin{equation}\label{equ:intno}
  \langle m_i,n_j\rangle
  =\begin{cases}
    -(d+1) &\text{if $i=j$}\\
    1 &\text{if $i\ne j$}
  \end{cases}
\end{equation}
We can view $z_0,\dots,z_{d+1}$ as homogeneous coordinates on $\P^{d+1}$. 
For any $m\in M$, $z^m$ is a rational function on $\P^{d+1}$. If $m\in\Delta\cap M$, then $z^m$ can be viewed as a global section of $\cO(d+2)=\cO(-K_{\P^{d+1}})$, in the sense that $\div(z^m)-K_{\P^{d+1}}\ge 0$. More generally, for any $r\ge1$, the set 
\[
  \{z^m\mid m\in r\Delta\cap M\}
\]
is a basis for $\Hnot(\P^{d+1},\cO(r(d+2)))$.

There is an alternative description in which a global section of $\cO(r(d+2))$ is given as a homogeneous polynomial in the $z_i$ of degree $r(d+2)$. Given $m\in r\Delta\cap M$, define a monomial
\[
  \chi^{r,m}:=z^m\prod_{i=0}^{d+1}z_i^r.
\]
Then $(\chi^{r,m})_{m\in r\Delta\cap M}$ is a basis of the space of homogeneous polynomials of degree $r(d+2)$ in the $z_i$, and hence a basis for $\Hnot(\P^{d+1},\cO(r(d+2)))$. Note that the sections $\chi^{r,rm_i}=z_i^{r(d+2)}$, $0\le i\le d+1$, have no common zeros.
%
% 
%%%%%%%%%%%%%%%%%%%%%%%%%%%%%%%%%%%%%%%%%%%%%%%%%%%%%%%%%%%%%%%%%%%
%
%
\section{Tropical manifolds}\label{sec:trop}
Above we defined simplices $\Delta\subset M_\R$ and $\Delta^\vee\subset N_\R$. Their boundaries
\begin{equation*}
  \Aspace:=\partial\Delta
  \quad\text{and}\quad
  \Bspace:=\partial\Delta^\vee
\end{equation*}
will be key players in what follows. As we will see, they are integral tropical manifolds in the sense of~\cite{GS06}. The exposition below more or less follows~\cite{LiFermat}.

The spaces $\Aspace$ and $\Bspace$ are naturally equipped with piecewise integral affine structures, and hence a canonical volume form that we refer to as Lebesgue measure. The total mass of $\Aspace$ and $\Bspace$ is $|\Aspace|=(d+2)^{d+1}/d!$ and $|\Bspace|=(d+2)/d!$, respectively.
It will occasionally be convenient to parametrize $\Aspace$ and $\Bspace$ as follows:
\begin{align}
    \Aspace &=\{\sum_j\a_jm_j\mid \a_j\in\R, \min_j\a_j=0, \sum_j\a_j=1\}\label{equ:Apar}\\
    \Bspace &=\{\sum_j\b_jn_j\mid \b_j\in\R, \min_j\b_j=0, \sum_j\b_j=1\}.\label{equ:Bpar}
\end{align}

% 
%%%%%%%%%%%%%%%%%%%%%%%%%%%%%%%%%%%%%%%%%%%%%%%%%%%%%%%%%%%%%%%%%%%
%
\subsection{Singular integral affine structure}\label{sec:SIAS}
Following~\cite{GS06,LiFermat}, we now upgrade the piecewise integral structures on $\Aspace$ and $\Bspace$ to singular integral affine structures. This means that we have open dense subsets $\Aspace_0\subset\Aspace$ and $\Bspace_0\subset\Bspace$, of real codimension 2, such that $\Aspace_0$ and $\Bspace_0$ each admit a sheaf of integral affine functions.

In general, there is a great deal of flexibility in the choice of $\Aspace_0$ and $\Bspace_0$, see e.g.~\cite{MP21}. We will, however, be interested in symmetric data on $\Aspace$ and $\Bspace$, i.e.\ data invariant under the action of the permutation group $G=S_{d+2}$ on $\Aspace$ and $\Bspace$. % (the action is induced by the natural action of $G$ on $M'=\Z^{d+2}$).
This gives a canonical choice of our singular set, namely, the barycentric complexes of the $(d-1)$-dimensional faces of $\Aspace$ and $\Bspace$. 

Let us now be more precise. First consider the $d$-dimensional faces of $\Aspace$ and $\Bspace$. These are of the form 
\[
  \sigma_i:=\{\max_jn_j=n_i=1\}\subset M_\R
  \quad\text{and}\quad
  \tau_i:=\{\max_jm_j=m_i=1\}\subset N_\R
\]
for $0\le i\le d+1$, and we write $\sigma_i^\circ$, $\tau_i^\circ$ for the relative interiors. The integral affine functions on $\sigma_i^\circ$ (resp.\ $\tau_i^\circ$) are the restrictions of the integral affine functions on $M_\R$ (resp.\ $N_\R$).

\begin{center}
\begin{figure}[H]
\begin{tikzpicture}[scale=0.6]%%%%
%\coordinate (2) at (0,0);
\coordinate (4) at (0,4);
\coordinate (3) at (3,2);
\coordinate (1) at (-3,0);
\coordinate (5) at (2,-2);
\coordinate (03) at (-0.5,-1);
\coordinate (13) at (0,1);
\coordinate (013) at (2/3,0);
\filldraw (3) circle (1pt);
\filldraw (4) circle (1pt);
\filldraw (5) circle (1pt);
\filldraw (1) circle (1pt);
\node[right] at (3) {$n_3$};
\node[above] at (4) {$n_2$};
\node[below] at (5) {$n_0$};
\node[left] at (1) {$n_1$};
\fill[yellow, fill opacity=0.15] (5) -- (3) -- (4);
\draw (3)--(4)--(1)--(5)--(3);
\draw[dashed] (1)--(3);
\draw (1)--(4)--(5);
\end{tikzpicture}
\caption{Subset $\tau_1$ for $d=2$}
\end{figure}
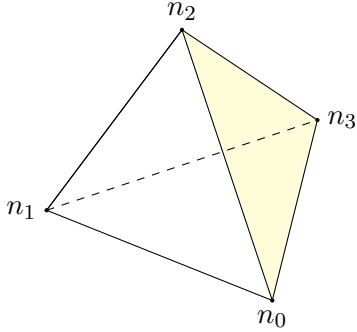
\end{center}
Second, we can define the integral affine structure near vertices of $\Aspace$ and $\Bspace$, respectively. Let $\Star(m_i)=\bigcup_{j\ne i}\sigma_j$ be the closed star of $m_i$, and $\Star^\circ(m_i)=\Aspace\setminus\sigma_i$ the open star. The stars $\Star(n_i)$ and $\Star^\circ(n_i)$ are defined analogously.

As follows from~\eqref{equ:intno}, given $i\ne j$, the integral linear map $M_\R\to\R^d$ given by
\begin{equation}\label{equ:pijinv}
  m\mapsto \left(\langle m, n_j - n_k\rangle\right)_{k\not= i, j}
\end{equation}
restricts to a piecewise integral affine isomorphism $\Star(m_i)\simto\tilde S$,
where $\tilde S\subset\R^d$ is the simplex with vertices given by
\[
  (d+2,0,\dots,0),\dots,(0,\dots,d+2),\ \text{and}\ (-(d+2),\dots,-(d+2)).
\]
It will be notationally convenient to denote the map in~\eqref{equ:pijinv} by $$p^{-1}_{i,j}\colon\Star(m_i)\simto\tilde S.$$ In this way, the inverse $p_{i,j}\colon \tilde S\to\Star(m_i)$ is an integral piecewise affine isomorphism, whose restriction to any simplex spanned by the origin and $d$ of the vertices of $\tilde S$ above is an integral affine isomorphism onto a simplex $\sigma_k$, $k\ne i$, when $\sigma_k$ is endowed with the integral affine structure above.
We view $p_{i,j}^{-1}$ as coordinates on $\Star(m_i)$.

By using Proposition~\ref{prop:chart_pair} below, one can easily check that for any $j, k, \ell\not= i$, the function $(n_k - n_\ell)\circ p_{i,j}: \ti{S}\rightarrow \R$ is the restriction of an integral linear function on $\R^d$. From this, it follows that $p_{i,k}^{-1}\circ p_{i,j}: \ti{S}\rightarrow \ti{S}$ is the restriction of an integral linear isomorphism of $\R^d$.

Similarly, we define coordinates on $\Star(n_i)$ by:
\begin{equation}\label{equ:qijinv}
  q_{i,j}^{-1}(n)
  = \left(\left\langle \frac{m_k - m_j}{d+2}, n\right\rangle\right)_{k\not= i, j}
  = \left(\left\langle e_k-e_j, n\right\rangle\right)_{k\not= i, j}\subset \R^{d}.
\end{equation}
Note the sign change, which makes the duality pairing in the charts compatible with the global pairing between $M_\R$ and $N_\R$, see Proposition~\ref{prop:chart_pair}. We get a piecewise integral affine isomorphism
\begin{equation*}
  q_{i,j}\colon\tilde T\simto\Star(n_i),
\end{equation*}
where $\tilde T\subset\R^d$ is the simplex spanned by
\[
  (-1,0,\dots,0),\dots,(0,\dots,-1),\ \text{and}\ (1,\dots,1).
\]
If $j,k\ne i$, then $q_{i,k}^{-1}\circ q_{i,j}:\tilde T\to\tilde T$ is the restriction of an integral linear isomorphism of $\R^d$, and 
for $j,k,l\ne i$, $(m_k-m_l)\circ q_{i,j}\colon\tilde T\to\R$ is the restriction of an integral linar function on $\R^d$.

As $p_{i,j}$ and $q_{i,j}$ are integral piecewise integral isomorphisms, they map Lebesgue measure on $\R^d$ to Lebesgue measure on $\Aspace$ and $\Bspace$, respectively. 

It is tempting to define integral affine structures on $\Star^\circ(m_i)$ and $\Star^\circ(n_i)$ by pulling back the sheaf on integral affine functions on $\tilde S^\circ$ and $\tilde T^\circ$, respectively. 
 However, these sheaves don't agree on the overlaps;
we need to define branch cuts in the above charts in order to work globally on $\Aspace$ and $\Bspace$.
This corresponds to choosing the singular part of the singular affine structure, which again we will canonically choose to be the barycentric complex of the $(d-1)$-dimensional faces.

To describe this explicitly, define subsets $S_i\subset\Star(m_i)$ and $T_i\subset\Star(n_i)$ by
\begin{equation*}
    S_i:=\{n_i=\min_jn_j\}
    \quad\text{and}\quad
    T_i:=\{m_i=\min_jm_j\}.
\end{equation*}
\begin{center}
\begin{figure}[H]
\begin{tikzpicture}[scale=0.6]%%%%
\coordinate (4) at (0,4);
\coordinate (3) at (3,2);
\coordinate (1) at (-3,0);
\coordinate (5) at (2,-2);
\coordinate (03) at (-0.5,-1);
\coordinate (13) at (0,1);
\coordinate (013) at (2/3,0);
\coordinate (012) at (5/3,4/3);
\coordinate (023) at (-1/3,2/3);
\coordinate (123) at (0,2);
\coordinate(23) at (-3/2,2);
\coordinate (02) at (1,1);
\coordinate (12) at (3/2,3);
\coordinate (01) at (5/2,0);
\filldraw (3) circle (1pt);
\filldraw (4) circle (1pt);
\filldraw (5) circle (1pt);
\filldraw (1) circle (1pt);
\node[right] at (3) {$n_3$};
\node[above] at (4) {$n_2$};
\node[below] at (5) {$n_0$};
\node[left] at (1) {$n_1$};
\draw[dotted] (023)--(02);
\draw[dotted] (013)--(01);
\draw[dotted] (123)--(12);
\draw(03)--(013)--(13);
\draw (03)--(023)--(23);
\draw (13)--(123)--(23);
\fill[red, fill opacity=0.15] (1) -- (03) -- (023)--(23);
\fill[orange, fill opacity=0.3] (1) -- (23) -- (123)--(13);
\fill[orange, fill opacity=0.2] (1) -- (03) -- (013)--(13);
\draw (3)--(4)--(1)--(5)--(3);
\draw[dashed] (1)--(3);
\draw (1)--(4)--(5);
\end{tikzpicture}
\caption{Subset $T_1$ for $d=2$}
\end{figure}
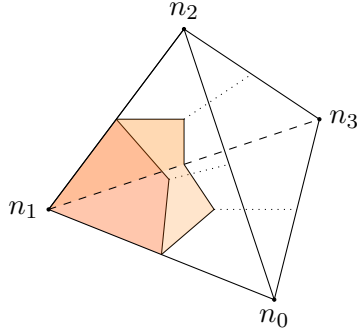
\end{center}
Their relative interiors are given by $S_i^\circ=\{n_i<\min_{j\ne i}n_j\}$ and $T_i^\circ=\{m_i<\min_{j\ne i}m_j\}$, respectively, and are open neighborhoods of $m_i$ and $n_i$ in $\Aspace$ and $\Bspace$, respectively. Note that $S_i^\circ\cap S_j^\circ=\emptyset$ and $T_i^\circ\cap T_j^\circ=\emptyset$ if $i\ne j$. We can easily describe these sets in terms of the parametrizations~\eqref{equ:Apar} and~\eqref{equ:Bpar}; for example,
\begin{equation*}
  S_i = \left\{\sum_j\a_jm_j\mid \ \a_i\ge\max_{j\ne i}\a_j\ge\min_{j\ne i}\a_j=0,\ \sum_j\a_j=1\right\}.
\end{equation*}

We now define the integral affine structure on $S_i^\circ$ and $T_i^\circ$ as the pullback of the integral affine structures on $\R^d$ under the maps $p_{i,j}^{-1}$ and $q_{i,j}^{-1}$, respectively. This is compatible with the integral affine structure on the open simplices $\sigma_l^\circ$ and $\tau_l^\circ$ as above. Moreover, the integral affine structures on $S_i^\circ$ and $S_j^\circ$ (resp.\ $T_i^\circ$ and $T_j^\circ$) are trivially compatible for $i\ne j$, since $S_i^\circ\cap S_j^\circ=\emptyset$ (resp.\ $T_i^\circ\cap T_j^\circ=\emptyset$).
We therefore obtain integral affine structures on 
\begin{equation*}
    \Aspace_0 := \bigcup_i\sigma_i^\circ\cup\bigcup_iS_i^\circ
    \quad\text{and}\quad
    \Bspace_0 := \bigcup_i\tau_i^\circ\cup\bigcup_iT_i^\circ,
\end{equation*}
and $\Aspace\setminus\Aspace_0$, $\Bspace\setminus\Bspace_0$ have codimension two.

% 
%%%%%%%%%%%%%%%%%%%%%%%%%%%%%%%%%%%%%%%%%%%%%%%%%%%%%%%%%%%%%%%%%%%
%
\subsection{Pairing and symmetries}
The pairing $M_\R\times N_\R\to\R$ restricts to a pairing
\[
  \Aspace\times\Bspace\to\R.
\]
Given $m\in\Aspace$ and $n\in\Bspace$, write $m=\sum_{j=0}^{d+1}\a_jm_j$ and $n=\sum_{j=0}^{d+1}\b_jn_j$, where $\min_j\a_j=\min\b_j=0$ and $\sum_j\a_j=\sum_j\b_j=1$. Using~\eqref{equ:intno} we then have
\begin{equation}\label{equ:pairing}
  \langle m,n\rangle=1-(d+2)\sum_j\a_j\b_j.
\end{equation}

In~\S\ref{sec:sym} it will be important to understand how the pairing interacts with the action of the permutation group $G=S_{d+2}$ on $M'=\Z^{d+2}$, and its various induced actions. Note that $G$ acts on the sets of simplices $\sigma_i$, $\tau_i$ and stars $\Star(m_i)$, $\Star(n_i)$, $S_i$, $T_i$, mapping relative interiors to relative interiors. We also have $\langle g(m),g(n)\rangle=\langle m,n\rangle$ for $m\in\Aspace$, $n\in\Bspace$, but not always $\langle m,g(n)\rangle=\langle m,n\rangle$.  
\begin{lem}\label{lem:pairsym2}
  Pick any $m\in\Aspace$, $n\in\Bspace$, and let $G(m,n)\subset G$ be the set of $g\in G$ such that $\langle m,g(n)\rangle$ is maximal. Then, for any $i\in\{0,1,\dots,d+1\}$ we have:
  \begin{itemize}
  \item[(i)]
    if $m\in\sigma_i$ (resp $m\in S_i$), then $g(n)\in T_i$ (resp. $g(n)\in\tau_i$) for some $g\in G(m,n)$;
  \item[(i')]
    if $m\in\sigma^\circ_i$ (resp $m\in S_i^\circ$), then $g(n)\in T_i$ (resp. $g(n)\in\tau_i$) for all $g\in G(m,n)$;
  \item[(ii)]
    if $n\in\tau_i$ (resp $n\in T_i$), then $g(m)\in S_i$ (resp. $g(m)\in\sigma_i$) for some $g\in G(m,n)$;
  \item[(ii')]
    if $n\in\tau_i^\circ$ (resp $n\in T_i^\circ$), then $g(m)\in S_i$ (resp. $g(m)\in\sigma_i$) for all $g\in G(m,n)$.
  \end{itemize}
\end{lem}

\begin{proof}
It suffices to prove~(i) and~(i'); the proofs of~(ii) and~(ii') are analogous.
Write $m=\sum_j\a_jm_j$ and $n=\sum_j\b_jn_j$, with $\min_j\a_j=\min_j\b_j=0$ and $\sum_j\a_j=\sum_j\b_j=1$.

To prove~(i), suppose $m\in\sigma_i$ (resp.\ $m\in S_i$), so that $\a_i=0$ (resp.\ $\a_i=\max_l\a_l$). Pick any $g'\in G(m,n)$, and choose $j$ such that $g'(n)\in T_j$ (resp.\ $g'(n)\in\tau_j$), that is, $\b_{g'^{-1}(j)}=\max_l\b_l$ (resp.\ $\b_{g'^{-1}(j)}=0$). Set $g=h\circ g'$, where $h\in G$ is the transposition of $\{0,1,\dots,d+1\}$ exchanging $i$ and $j$. Then $g(n)\in T_i$ (resp.\ $g(n)\in\tau_i$), and we claim that $g\in G(m,n)$. But~\eqref{equ:pairing} implies
\begin{equation*}
  \langle m,g(n)-g'(n)\rangle
  =(d+2)(\a_j-\a_i)(\b_{g'^{-1}(j)}-\b_{g'^{-1}(i)})\ge0.
\end{equation*}    

The proof of~(i') is similar.
Assume $m\in\sigma_i^\circ$ (resp.\ $m\in S_i^\circ$), so that $\min_{j\ne i}\a_j>\a_i=0$ (resp.\ $\a_i>\max_{j\ne i}\a_j$). It suffices to prove that if $n\not\in  T_i$ (resp.\ $n\not\in\tau_i$), then there exists $g\in G$ such that $\langle m,g(n)-n\rangle>0$. But $n\not\in T_i$ (resp.\ $n\not\in\tau_i$) means that $\b_j>\b_i$ for some $j$ (resp.\ $\b_i>0$). Let $g\in G$ be transposition exchanging $i$ and $j$. Then
\begin{equation*}
  \langle m,g(n)-n\rangle
  =(d+2)(\a_j-\a_i)(\b_j-\b_i)>0,
\end{equation*}
completing the proof.
\end{proof}
\color{black}
% 
%%%%%%%%%%%%%%%%%%%%%%%%%%%%%%%%%%%%%%%%%%%%%%%%%%%%%%%%%%%%%%%%%%%
%
\subsection{Pairing in coordinate charts}
Lemma~\ref{lem:pairsym2} suggests that the pairing between $\Aspace$ and $\Bspace$ is most natural between $\sigma_i$ and $\mathrm{Star}(n_i)$, or between $\Star(m_i)$ and $\tau_i$. We now calculate the pairing between elements in compatible coordinate charts defined on these regions.

\begin{comment}
\begin{prop}\label{prop:chart_pair}
Fix indices $i\not= j$. For $s\in p_{j,i}^{-1}(\sigma_i)$ and $t\in\tilde T= q_{i,j}^{-1}(\mathrm{Star}(n_i))$, we have:
\begin{equation}\label{equ:cp1}
\langle s, t\rangle = \langle p_{j,i}(s)-m_j, q_{i,j}(t)\rangle.
\end{equation}
Similarly, for all $s\in \tilde S=p_{i,j}^{-1}(\mathrm{Star}(m_i))$ and $t\in  q_{j,i}^{-1}(\tau_i)$, we have
\begin{equation}\label{equ:cp2}
\langle s, t\rangle = \langle p_{i,j}(s), q_{j,i}(t)-n_j\rangle.
\end{equation}
\end{prop}
\end{comment}

\corr{
\begin{prop}\label{prop:chart_pair}
Fix indices $i\not= j$. For $x\in p_{j,i}^{-1}(\sigma_i)$ and $y\in\tilde T= q_{i,j}^{-1}(\mathrm{Star}(n_i))$, we have:
\begin{equation}\label{equ:cp1}
\langle x, y\rangle = \langle p_{j,i}(x)-m_j, q_{i,j}(y)\rangle.
\end{equation}
Similarly, for all $x\in \tilde S=p_{i,j}^{-1}(\mathrm{Star}(m_i))$ and $y\in  q_{j,i}^{-1}(\tau_i)$, we have
\begin{equation}\label{equ:cp2}
\langle x, y\rangle = \langle p_{i,j}(x), q_{j,i}(y)-n_j\rangle.
\end{equation}
\end{prop}
}

Note that the pairing on the right-hand sides of \eqref{equ:cp1}, \eqref{equ:cp2} is between $M_\R$ and $N_\R$, while the pairing on the left-hand side is the scalar product on $\R^d$.

\begin{proof}
  Pick $m\in\sigma_i\subset\Star(m_j)$ and $n\in\Star(n_i)$. Write 
  $m = \sum_{k\ne i} \alpha_km_k$ and $n = \sum_{k}\beta_kn_k$, where $\a_k,\b_k\ge0$ and $\sum_k\a_k=\sum_k\b_k=1$. Then $m-m_j=\sum_{k\ne i,j}\a_k(m_k-m_j)$, so that
  \begin{equation*}
    \langle m-m_j,n\rangle=(d+2)\sum_{k\ne i,j}\a_k(\b_j-\b_k)
  \end{equation*}
  in view of~\eqref{equ:intno}.  On the other hand,~\eqref{equ:pijinv} and~\eqref{equ:qijinv} give
\[
  p_{j,i}^{-1}(m) = (d+2)(\alpha_k )_{k\not=i,j}
  \quad\text{and}\quad
  q_{i,j}^{-1}(n) = ( \beta_j - \beta_k)_{k\not=i,j},
\]
which implies $\langle m-m_j,n\rangle=\langle p_{j,i}^{-1}(m), q_{i,j}^{-1}(n)\rangle$. We now obtain~\eqref{equ:cp1} by inverting the coordinate maps, and~\eqref{equ:cp2} is proved in the same way.
\end{proof}
% 
%
%%%%%%%%%%%%%%%%%%%%%%%%%%%%%%%%%%%%%%%%%%%%%%%%%%%%%%%%%%%%%%%%%%%
%
%
\section{The c-transform and the class of c-convex functions}\label{sec:cconvex}
Denote by $L^\infty(\Aspace)$ and $L^\infty(\Bspace)$ the space of bounded real-valued functions on $\Aspace$ and $\Bspace$, respectively.
%
%%%%%%%%%%%%%%%%%%%%%%%%%%%%%%%%%%%%%%%%%%%%%%%%%%%%%%%%%%%%%%%%%%%
%
\subsection{General definitions and properties}
We start by defining the \emph{c-transforms}
\begin{equation*}
  L^\infty(\Aspace)\to L^\infty(\Bspace)
  \quad\text{and}\quad
  L^\infty(\Bspace)\to L^\infty(\Aspace)
\end{equation*}
as follows.
 Given $\phi\in L^\infty(\Aspace)$, we define a new function $\phi^c\in L^\infty(\Bspace)$ by
\begin{equation}\label{equ:ctransf1}
  \phi^c(n):=\sup_{m\in\Aspace}\langle m,n\rangle-\phi(m).
\end{equation}
Note that $\phi^c$ is bounded since $-(d+1)\le\langle m,n\rangle\le 1$.
Similarly, given $\psi\in L^\infty(\Bspace)$, we define $\psi^c\in L^\infty(\Aspace)$ by
\begin{equation}\label{equ:ctransf2}
  \psi^c(m):=\sup_{n\in\Bspace}\langle m,n\rangle-\psi(n).
\end{equation}
\begin{rmk}\label{rmk:cost}
  The c-transform in this setting is inspired by the usual one in optimal transport~\cite{AG13}, and can be defined much more generally, e.g.\ when $X=\Aspace$ and $Y=\Bspace$ are replaced by arbitrary sets, and $\langle m,n\rangle$ by an arbitrary `cost' function $c\colon X\times Y\to\R$. In that generality, $\phi^c$ and $\psi^c$ may take infinite values, but our cost function is uniformly bounded, so we can restrict to bounded functions. 
\end{rmk}
It is a general fact that $(\phi+a)^c=\phi^c-a$ and $(\psi+a)^c=\psi^c-a$ for any bounded functions $\phi,\psi$ and any constant $a$. Moreover, if $\phi_1\le\phi_2$, then $\phi_1^c\ge\phi_2^c$, and similarly for the c-transform in the other direction. This formally implies that the c-transforms are contractive: $\|\phi_1^c-\phi_2^c\|\le\|\phi_1-\phi_2\|$ and $\|\psi_1^c-\psi_2^c\|\le\|\psi_1-\psi_2\|$ for $\phi_i\in L^\infty(\Aspace)$ and $\psi_i\in L^\infty(\Bspace)$, where $\|\cdot\|$ denotes the $\sup$ norm.

In our case, we also have $0^c=1$, as follows from $\max_{m\in\Aspace}\langle m,n\rangle=1$ for all $n\in\Bspace$ and $\max_{n\in\Bspace}\langle m,n\rangle=1$ for all $m\in\Aspace$.

\begin{lem}\label{lem:2c}
  For any bounded functions $\phi\colon\Aspace\to\R$ and $\psi\colon\Bspace\to\R$, we have $\phi^{cc}\le\phi$, $\psi^{cc}\le\psi$, $\phi^{ccc}=\phi^c$, and $\psi^{ccc}=\psi^c$.
\end{lem}
\begin{proof}
  This is formal, see~\cite[p.8]{AG13}.
\end{proof}
\begin{defi}\label{defi:PandQ}
  We define $\cP\subset L^\infty(\Aspace)$ and $\cQ\subset L^\infty(\Bspace)$ as the images of the c-transform,
\[
\mathcal{P} := \{\phi = v^c\ |\ v\in L^\infty(\Bspace)\}\quad\text{and}\quad \mathcal{Q} := \{\psi = u^c\ |\ u\in L^\infty(\Aspace)\},
\]
and equip $\cP$ and $\cQ$ with the supremum norm. 
\end{defi}
The functions in $\cP$ and $\cQ$ are called \emph{c-convex}.
It follows from the remarks above that the spaces $\cP$ and $\cQ$ of c-convex functions are invariant under the addition of a real constant, and they consist of bounded functions. They also contain all constant functions.
\begin{lem}
  The c-transform defines isometric bijections $\cP\to\cQ$ and $\cQ\to\cP$ that are inverse to each other.
\end{lem}
\begin{proof}
  By Lemma~\ref{lem:2c}, the two maps are bijective, and inverse to one another. As they are both contractive, they must be isometries.
\end{proof}

\begin{lem}\label{lem:Lip}
  The functions in $\cP$ and $\cQ$ are uniformly Lipschitz continuous.
\end{lem}
\begin{proof}
  Suppose $\psi$ is a bounded function on $\Bspace$. By definition, $\psi^c(m) = \sup_{n\in \Bspace}(\langle m,n\rangle-\psi(n))$; this defines a locally bounded function on $M_{\R}$. Each of the $\langle m, n\rangle - \psi(n)$ is linear, with uniform Lipschitz constant, since $\Bspace$ is compact. It follows that $\psi^c(m)$ is also Lipschitz on $M_\R$, with the same constant. 
  The same argument obviously works for $\cQ$.
\end{proof}

\begin{cor}\label{cor:compact}
  The spaces $\cP$ and $\cQ$ are closed subspaces of $C^0(\Aspace)$ and $C^0(\Bspace)$, respectively. Moreover, $\cP/\R$ and $\cQ/\R$ are compact.
\end{cor}
\begin{proof}
  Lemma~\ref{lem:Lip} shows that $\cP\subset C^0(\Aspace)$. To prove that $\cP$ is a closed subspace, consider a sequence $(\phi_k)_1^\infty$ in $\cP$ converging uniformly to $\phi\in C^0(\Aspace)$. Then $\phi_k^{cc}=\phi_k$ for all $k$, so since the double c-transform is a continuous (even contractive) map from $C^0(\Aspace)\to\cP$, we must have $\phi^{cc}=\phi$, so that $\phi\in\cP$. Thus $\cP$ is closed.

  To prove that $\cP/\R$ is compact, it suffices to show that the closed subspace $\cP_0:=\{\phi\in\cP\mid\max\phi=0\}$ is compact. But Lemma~\ref{lem:Lip} shows that the functions in $\cP_0$ are uniformly bounded, and equicontinuous, so we conclude using the Arzel\`a--Ascoli theorem. 

The same argument shows that $\cQ\subset C^0(\Bspace)$ is closed and that $\cQ/\R$ is compact.
\end{proof}
\begin{rmk}\label{rmk:partialc}
 For any subset $\Aspace'\subset\Aspace$ and any bounded function $\phi\colon\Aspace'\to\R$, the function $\psi\colon\Bspace\to\R$ defined by $\psi=\sup_{m\in\Aspace'}(m-\phi(m))$ is c-convex. Indeed, $\psi$ is the c-transform of the extension of $\phi$ to $\Aspace$ defined by 
 $\phi|_{\Aspace\setminus\Aspace'}\equiv\sup_{\Aspace'}\phi+d+2$.
\end{rmk}

Lastly, we have the following definition, also standard in the optimal transport literature:
\begin{defi}
  Given $\psi\in\cQ$, the \emph{c-subgradient} of $\psi$ is the multi-valued map $\partial^c\psi\colon\Bspace\to\Aspace$ given by 
\begin{equation*}
  (\partial^c\psi)(n):=\{m\in\Aspace\mid \psi(n)+\psi^c(m)=\langle m,n\rangle\}
\end{equation*}
for any $n\in\Bspace$. 
\end{defi}
By continuity, the c-subgradient is nonempty. When it is a singleton, we call it a \emph{c-gradient}. We make similar definitions for $\phi\in\mathcal{P}$. It is evident that for $\p\in\cQ$, $m\in\Aspace$ and $n\in\Bspace$, we have $m\in(\partial^c\psi)(n)$ iff  $n\in(\partial^c\psi^c)(m)$, so that $\partial^c\psi$ and $\partial^c\psi^c$ are inverses, in the sense of multi-valued maps.
\begin{exam}
Let $\psi = \max_i m_i \equiv 1 \in \cQ_{\sym}$ where the max is taken over the vertices of $A$. Then  
$$\partial^c\psi^c(n) = \{ m\in A: \langle m,n \rangle = 1\} $$
is the face in $A$ dual to the smallest face in $B$ containing $n$.
\end{exam}

% 
%%%%%%%%%%%%%%%%%%%%%%%%%%%%%%%%%%%%%%%%%%%%%%%%%%%%%%%%%%%%%%%%%%%
%
\subsection{Extension property}\label{sec:extend}
In~\cite{LiFermat}, Li studies the class of functions on $\Aspace$ and $\Bspace$ which satisfy what he calls the extension property, motivated in part by extension theorems for (quasi-)plurisubharmonic functions: see e.g.~\cite{CGZ13, CT14, WZ20, NWZ21, CGZ22}.
Here, similarly to~\cite[Proposition~3.19]{LiFermat}, we show that these extendable functions are exactly those in $\cP$ and $\cQ$, and discuss their canonical extensions to $M_{\R}$ and $N_{\R}$. 

We set some notation. As in~\cite{BB13}, let $\cP_+$ be the set of convex functions $\phi\colon M_\R\to\R$ such that $\phi=\max_jn_j+O(1)$, and $\cQ_+$ the set of convex functions $\psi\colon N_\R\to\R$ with $\psi=\max_jm_j+O(1)$.  Using~\eqref{equ:ctransf1} and~\eqref{equ:ctransf2}, we can view the c-transforms as maps
\begin{equation*}
  L^\infty(\Aspace)\to\cQ_+
  \quad\text{and}\quad
  L^\infty(\Bspace)\to\cP_+,
\end{equation*}
so that all functions in $\mathcal{P}$ and $\mathcal{Q}$ come from restrictions of functions in $\mathcal{P}_+$ and $\mathcal{Q}_+$. The following proposition shows the converse:
\begin{prop}\label{prop:extend}
  Suppose that $\psi\in\mathcal{Q}_+$. Then $\psi^{cc} \geq \psi$ on $N_{\R}\setminus (\Delta^{\ch})^{\circ}$. % and $\psi^{cc} \leq \psi$ on $\Delta^{\ch}$.
 It follows that $\psi^{cc} = \psi$ on $\Bspace=\partial\Delta^\vee$. The corresponding statements hold for $\phi\in\mathcal{P}_+$.
\end{prop}
\begin{proof}
  It suffices to prove $\psi^{cc}\ge\psi$ on $N_\R\setminus(\Delta^\vee)^\circ$. Indeed, the inequality  $\psi^{cc}\le\psi$ on $\Bspace$ is formal, see Lemma~\ref{lem:2c}.

  Pick any $n_0\in N_{\R}\setminus (\Delta^{\ch})^\circ$. To see that $\psi(n_0) \leq \psi^{cc}(n_0)$, it will suffice to find an $m\in \Aspace$ such that:
\[
\psi(n_0) \leq \langle m, n_0\rangle - \psi^c(m),
\]
since the right-hand side is dominated by $\psi^{cc}(n_0)$. 
Let $m'$ be a subgradient of $\psi$ at $n_0$, i.e.
%\in \partial\psi(n_0)\subseteq \Delta$; then, by definition, we have:
\[
\psi(n) \geq \langle m', n - n_0\rangle + \psi(n_0),
\]
for all $n\in N_{\R}$. Since $\psi\in\mathcal{Q}_+$, the subgradients for $\psi$ satisfy $\partial\psi (N_{\R}) \subseteq \Delta$ (see e.g.~\cite[Lemma 2.5]{BB13}). Also, as $n_0$ is not in the interior of $\Delta^{\ch}$, we can find a hyperplane, represented by $m_0\in M_{\R}$, such that $\sup_{n\in\Bspace} \langle m_0, n\rangle \leq \langle m_0, n_0\rangle$.

\begin{comment}
    Now let $t \geq 0$ be such that $m := m' + tm_0\in \Aspace$. Then we have that:
\begin{multline*}
  \psi^{c}(m)
  = \sup_{n\in\Bspace} \langle m, n\rangle - \psi(n)
\leq -\psi(n_0) + \sup_{n\in\Bspace} \langle m' + tm_0, n\rangle + \langle m', n_0 - n\rangle\\
= \langle m', n_0\rangle -\psi(n_0) + t\sup_{n\in\Bspace} \langle m_0, n\rangle
\leq \langle m, n_0\rangle - \psi(n_0),
\end{multline*}
and we are done.
\end{comment}
\corr{
Now let $\lambda \geq 0$ be such that $m := m' + \lambda m_0\in \Aspace$. Then we have that:
\begin{multline*}
  \psi^{c}(m)
  = \sup_{n\in\Bspace} \langle m, n\rangle - \psi(n)
\leq -\psi(n_0) + \sup_{n\in\Bspace} \langle m' + \lambda m_0, n\rangle + \langle m', n_0 - n\rangle\\
= \langle m', n_0\rangle -\psi(n_0) + \lambda\sup_{n\in\Bspace} \langle m_0, n\rangle
\leq \langle m, n_0\rangle - \psi(n_0),
\end{multline*}
and we are done.
}
\end{proof}
\begin{cor}\label{cor:PQconvex}
  The spaces $\cP$ and $\cQ$ are convex.
\end{cor}
\begin{proof}
This is clear since the spaces $\cP_+$ and $\cQ_+$ are convex.
\end{proof}
\begin{rmk}
  Unlike the plurisubharmonic case, functions in $\cP$ (resp.\ $\cQ$) admit a canonical extension to $M_\R$ (resp.\ $N_\R$), namely the supremum of all such extensions. We omit the proof.
\end{rmk}

% 
%%%%%%%%%%%%%%%%%%%%%%%%%%%%%%%%%%%%%%%%%%%%%%%%%%%%%%%%%%%%%%%%%%%

% 
%%%%%%%%%%%%%%%%%%%%%%%%%%%%%%%%%%%%%%%%%%%%%%%%%%%%%%%%%%%%%%%%%%%

% 
%%%%%%%%%%%%%%%%%%%%%%%%%%%%%%%%%%%%%%%%%%%%%%%%%%%%%%%%%%%%%%%%%%%
%
\subsection{Convexity in coordinate charts}
Following Li~\cite{LiFermat}, we show that the functions in $\mathcal{P}$ and $\mathcal{Q}$ are convex in the coordinate charts defined in~\S\ref{sec:SIAS}, up to adding a piecewise linear term.
\begin{lem}\label{lem:cstar}
\cite[Proposition~3.26]{LiFermat}
If $\psi\in\cQ$, then for any $i\ne j$, the function 
\[
\psi_{i,j} := (\psi-m_j)\circ q_{i,j}
\]
is convex on $q_{i,j}^{-1}(\Star(n_i))$. As a consequence, $\psi\circ q_{i,j}$ is convex on $q_{i,j}^{-1}(\tau_k)$ for any $k\ne i$.

Similarly, if $\phi\in\cP$, then for any $i\ne j$, the function 
\[
\phi_{i,j} := (\phi-n_j)\circ p_{i,j}
\]
is convex on $p_{i,j}^{-1}(\Star(m_i))$, and $\phi\circ p_{i,j}$ is convex on $p_{i,j}^{-1}(\sigma_k)$ for any $k\ne i$.
\end{lem}
In the terminology of~\cite{LiFermat}, the lemma says that the functions in $\cP$ and $\cQ$ are locally convex.
\begin{proof}
  We prove the statement about $\cQ$; functions in $\cP$ are handled in the same way. Thus pick $\psi\in\cQ$. 
  We shall in fact prove the following: if $n\in \tau_j\subset\Star(n_i)$ and $m\in \partial^c\psi(n)$, then $p_{j, i}^{-1}(m)$ is a subgradient for $\psi_{i,j}$ at $q_{i,j}^{-1}(n)$ (here we are thinking of $p_{j,i}^{-1}$ as a global map from $\Aspace$ to $\R^d$, and make no assumption on where $m$ is inside $\Aspace$). Accepting this, and noting that $\partial^c\psi(n)$ is non-empty for any $n\in \Bspace$ by compactness, it follows that $\psi_{i,j}$ is convex on $\Star(n_i)$. The proposition then follows by noting that $$\psi_{i,k} = \left(\psi_{i,j}+(m_j - m_{k})\circ q_{i,j}\right) \circ q_{i,j}^{-1}\circ q_{i,k},$$ 
  $(m_j - m_{k})\circ q_{i,j}$ is affine on $q_{i,j}^{-1}(\Star(n_i))$ for any $j,k \not= i$, and that the 
  maps $q_{i,j}^{-1}\circ q_{i,k}$ are linear on $q_{i,k}^{-1}(\Star(n_k))$.
  %the proposition follows by noting that $(m_k - m_{\ell})\circ q_{i,j}$ is affine on $q_{i,j}^{-1}(\Star(n_i))$ for any $k,\ell \not= i$, and that the 
  %maps $q_{i,j}^{-1}\circ q_{i,k}$ are linear.

First, from the definition of the $c$-subgradient, we have:
\[
\psi(n) = \langle m, n\rangle - \psi^{c}(m) \leq \langle m, n - n' \rangle + \psi(n')
\]
for all $n'\in \Star(n_i)$. \corr{With $y:=q_{i,j}^{-1}(n)$, $y':=q_{i,j}^{-1}(n')$, we have that
\[
  \psi_{i,j}(y')-\psi_{i,j}(y) \ge \langle m - m_j, n' - n\rangle,
\]
and it remains to estimate the right-hand side in terms of the coordinates.

We can write $m=m'+rm_i$, where $m'\in\sigma_i$ and $r\ge0$. Indeed, if $m=\sum_k\a_km_k$, then we can pick $m'=\sum_{k\ne i}(\a_k+\frac{\a_i}{d+1})m_k$ and $r=2\a_i$.
Now set $x:=p_{j,i}^{-1}(m)$, $x':=p_{j,i}^{-1}(m')$, and $x_i:=p_{j,i}^{-1}(m_i)$. 
By Proposition~\ref{prop:chart_pair}, we have
\[
  \langle x',y'-y\rangle=\langle m'-m_j,n'-n\rangle.
\]
On the other hand, a direct calculation as in the proof of Proposition~\ref{prop:chart_pair} yields
\[
  \langle x_i,y\rangle=\langle m_i,n\rangle + (d+1)\langle m_j,n\rangle
\quad\text{and}\quad
  \langle x_i,y'\rangle=\langle m_i,n'\rangle + (d+1)\langle m_j,n'\rangle.
\]
By linearity, $x=x'+rx_i$, and hence 
\begin{multline*}
  \langle x,y'-y\rangle
  =\langle m'-m_j,n'-n\rangle
  +r(\langle m_i,n'-n\rangle+(d+1)\langle m_j,n'-n\rangle)\\
  =\langle m-m_j,n'-n\rangle
  +r(d+1)\langle m_j,n'-n\rangle
  \le\langle m-m_j,n'-n\rangle,
\end{multline*}
where the inequality holds since $n\in\tau_j$ implies $\langle m_j,n\rangle=1\ge\langle m_j,n'\rangle$. Altogether, this yields
\[
  \psi_{i,j}(y')-\psi_{i,j}(y) \ge \langle x,y'-y\rangle,
\]
and completes the proof.
}
\begin{comment}
    With $t:=q_{i,j}^{-1}(n)$, $t':=q_{i,j}^{-1}(n')$, we have that
\[
  \psi_{i,j}(t')-\psi_{i,j}(t) \ge \langle m - m_j, n' - n\rangle,
\]
and it remains to estimate the right-hand side in terms of the coordinates.

We can write $m=m'+rm_i$, where $m'\in\sigma_i$ and $r\ge0$. Indeed, if $m=\sum_k\a_km_k$, then we can pick $m'=\sum_{k\ne i}(\a_k+\frac{\a_i}{d+1})m_k$ and $r=2\a_i$.
Now set $s:=p_{j,i}^{-1}(m)$, $s':=p_{j,i}^{-1}(m')$, and $s_i:=p_{j,i}^{-1}(m_i)$. 
By Proposition~\ref{prop:chart_pair}, we have
\[
  \langle s',t'-t\rangle=\langle m'-m_j,n'-n\rangle.
\]
On the other hand, a direct calculation as in the proof of Proposition~\ref{prop:chart_pair} yields
\[
  \langle s_i,t\rangle=\langle m_i,n\rangle + (d+1)\langle m_j,n\rangle
\quad\text{and}\quad
  \langle s_i,t'\rangle=\langle m_i,n'\rangle + (d+1)\langle m_j,n'\rangle.
\]
By linearity, $s=s'+rs_i$, and hence 
\begin{multline*}
  \langle s,t'-t\rangle
  =\langle m'-m_j,n'-n\rangle
  +r(\langle m_i,n'-n\rangle+(d+1)\langle m_j,n'-n\rangle)\\
  =\langle m-m_j,n'-n\rangle
  +r(d+1)\langle m_j,n'-n\rangle
  \le\langle m-m_j,n'-n\rangle,
\end{multline*}
where the inequality holds since $n\in\tau_j$ implies $\langle m_j,n\rangle=1\ge\langle m_j,n'\rangle$. Altogether, this yields
\[
  \psi_{i,j}(t')-\psi_{i,j}(t) \ge \langle s,t'-t\rangle,
\]
and completes the proof.
\end{comment}
\end{proof}

% 
%%%%%%%%%%%%%%%%%%%%%%%%%%%%%%%%%%%%%%%%%%%%%%%%%%%%%%%%%%%%%%%%%%%
%

\subsection{A principal $\R$-bundle}\label{sec:AffRb}
We can interpret the convexity statement in Lemma~\ref{lem:cstar} geometrically as follows. For $0\le j\le d+1$, set $Y_j:=N_\R$, and define a topological space $\Lambda$ by $\Lambda:=\coprod_jY_j\times\R/\!\sim$, \corr{where $(n,\lambda)\in Y_j\times\R$ and $(n',\lambda')\in Y_{j'}\times\R$ are equivalent iff $n=n'$ and $\lambda'-\lambda=\langle m_{j'}-m_j,n\rangle$.} The evident map $\pi\colon\Lambda\to N_\R$ gives $\Lambda$ the structure of a principal $\R$-bundle.
\footnote{One can also view $\Lambda$ as the skeleton of the analytification of the line bundle $\cO(d+2)$ over $\P^{d+1}$, restricted to $N_\R\subset\P^{d+1,\an}$, see~\cite[\S2.1]{trivval}.}

Let $Z\subseteq N_{\R}$. A continuous \emph{section} of $\Lambda$ over $Z$ is a continuous function $s\colon Z\to\Lambda$ such that $\pi\circ s=\id$. By construction, $\Lambda$ is trivial, and comes equipped with isomorphisms $\theta_j\colon\Lambda\simto Y_j\times\R$. These give rise to a canonical reference section $s_{\refe}$ over $N_\R$, defined by $\theta_j(s_{\refe}(n))=(n,\langle m_j,n\rangle)$. For any continuous section $s$ over $Z$, $s-s_{\refe}$ is a continuous function on $Z$. We set $s_j:=s_{\refe}+m_j$.

A continuous {\em metric} on $\Lambda$ over $Z$ can be viewed as a continuous function $\Psi:\pi^{-1}(Z)\rightarrow \R$ which respects the $\R$-action, i.e. $\Psi(s + r) = \Psi(s) + r$, for $s\in\pi^{-1}(Z)$, $r\in\R$. By checking its representations in coordinate charts, $\Psi$ is naturally a section of the ``dual" bundle $-\Lambda$; this is defined in exactly the same way as $\Lambda$, \corr{except we require $\lambda' - \lambda = \langle m_j - m_{j'}, n\rangle$.} It follows that $- s_{\refe}$ is a canonical reference metric on $\Lambda$.

The restriction of $\Lambda$ (and $-\Lambda$) to the integral affine manifold $\Bspace_0\subset N_\R$ can be equipped with the structure of an integral affine $\R$-bundle in the sense of~\cite{HO19}. Namely, we declare that, for any $i$, a continuous section $s$ of $\Lambda$ over $\tau_i^\circ$ (resp.\ $T_i^\circ$) is integral affine iff the function $s-s_j$ on $\tau_i^\circ$ (resp.\ $T_i^\circ$) is integral affine for some (equivalently, any) $j\ne i$.

Lemma~\ref{lem:cstar} now implies that for any $\psi\in\mathcal{Q}$, the metric $\Psi = \psi - s_{\refe}$ on $\Lambda$ is convex over $B_0$, since it is convex in any affine trivializations (equivalently, $\Psi$ is a convex section of $-\Lambda$~\cite{HO19}).

%
% 
%%%%%%%%%%%%%%%%%%%%%%%%%%%%%%%%%%%%%%%%%%%%%%%%%%%%%%%%%%%%%%%%%%%
%
%
\section{Symmetric c-convex functions and their Monge--Amp\`ere measures}\label{sec:sym}
The c-transform is modeled on the Legendre transform between convex functions on a vector space and its dual, and as shown in Lemma~\ref{lem:cstar}, leads to a seemingly satisfactory notion of ``local convexity" on $\Aspace$ and $\Bspace$. However, if one attempts to generalize Alexandrov's definition of the weak Monge--Amp\`{e}re measure to this setting, some interesting problems manifest.

As suggested by Li~\cite{LiFermat}, these issues disappear if we take into account the action of the permutation group $G=S_{d+2}$, and restrict ourselves to symmetric data.

%
% 
%%%%%%%%%%%%%%%%%%%%%%%%%%%%%%%%%%%%%%%%%%%%%%%%%%%%%%%%%%%%%%%%%%%
%
%
\subsection{Controlling the c-gradients}
We denote by $\cP_{\sym}\subset\cP$ and $\cQ_{\sym}$ the set of symmetric functions, that is, $G$-invariant functions. These are closed subsets of $\cP$ and $\cQ$, respectively, so the quotients $\cP_{\sym}/\R$ and $\cQ_{\sym}/\R$ are compact by Corollary~\ref{cor:compact}.
The c-transform is equivariant for the $G$-action, and restrict to isometric bijections between $\cP_{\sym}$ and $\cQ_{\sym}$.

As we now show, symmetry places a number of strong restrictions on the possible $c$-subgradients a function could have.
\begin{lem}\label{lem:pooh}
For any $\psi\in\cQ_{\sym}$, we have
$\partial^c\psi(T_i^\circ)\subseteq \sigma_i$ and $\partial^c\psi(\tau_i^{\circ})\subseteq S_i$. The analogous inclusions hold for $\phi\in\cP_{\sym}$.
\end{lem}
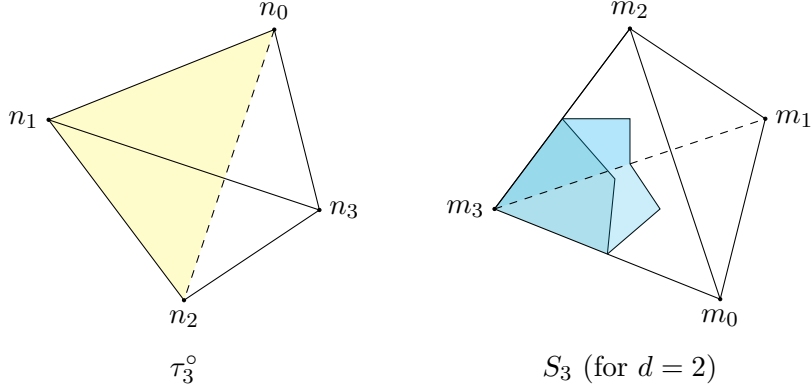
\begin{center}
\begin{figure}[H]
\begin{tikzpicture} [scale=0.6]%%%%
%\coordinate (2) at (0,0);
\coordinate (4) at (0,-4);
\coordinate (3) at (3,-2);
\coordinate (1) at (-3,0);
\coordinate (5) at (2,2);
\filldraw (3) circle (1pt);
\filldraw (4) circle (1pt);
\filldraw (5) circle (1pt);
\filldraw (1) circle (1pt);
\node[right] at (3) {$n_3$};
\node[below] at (4) {$n_2$};
\node[below] at (0,-5) {$\tau_3^\circ$};
\node[above] at (5) {$n_0$};
\node[left] at (1) {$n_1$};
\fill[yellow, fill opacity=0.2] (1) -- (5) -- (4);
\draw (3)--(4)--(1)--(5)--(3);
\draw[dashed] (5)--(4);
\draw (1)--(3);
\end{tikzpicture}
\quad \quad\begin{tikzpicture}[scale=0.6]%%%%
%\coordinate (2) at (0,0);
\coordinate (4) at (0,4);
\coordinate (3) at (3,2);
\coordinate (1) at (-3,0);
\coordinate (5) at (2,-2);
\coordinate (03) at (-0.5,-1);
\coordinate (13) at (0,1);
\coordinate (013) at (2/3,0);
\coordinate (023) at (-1/3,2/3);
\coordinate (123) at (0,2);
\coordinate(23) at (-3/2,2);
\filldraw (3) circle (1pt);
\filldraw (4) circle (1pt);
\filldraw (5) circle (1pt);
\filldraw (1) circle (1pt);
\node[right] at (3) {$m_1$};
\node[above] at (4) {$m_2$};
\node[below] at (5) {$m_0$};
\node[left] at (1) {$m_3$};
\node[below] at (0,-3) {$S_3$ (for $d=2$)};
\draw(03)--(013)--(13);
\draw (03)--(023)--(23);
\draw (13)--(123)--(23);
\fill[teal, fill opacity=0.15] (1) -- (03) -- (023)--(23);
\fill[cyan, fill opacity=0.3] (1) -- (23) -- (123)--(13);
\fill[cyan, fill opacity=0.2] (1) -- (03) -- (013)--(13);
\draw (3)--(4)--(1)--(5)--(3);
\draw[dashed] (1)--(3);
\draw (1)--(4)--(5);
\end{tikzpicture}
%\caption{Subset $\tau_2$}
\end{figure}
\end{center}
\begin{proof}
  By symmetry of $\psi$ and $\psi^c$, $m\in(\partial^c\psi)(n)$ implies $\langle m, n\rangle = \max_{g\in G} \langle g^{-1}(m), n\rangle$. The result now follows from Lemma~\ref{lem:pairsym2}.
\end{proof}
Since $\partial^c\psi$ and $\partial^c\psi^c$ are inverses, applying Lemma~\ref{lem:pooh} to $\psi^c$ gives:
\begin{cor}
For any $\psi\in\cQ_{\sym}$, we have $S_i^{\circ} \subseteq \partial^c\psi(\tau_i)$ and $\sigma_i^{\circ} \subseteq\partial^c\psi(T_i)$, with analogous results for $\phi\in\cP_{\sym}$.
\end{cor}

Next we look at the subgradients in charts. Recall that the function $\psi_{i,j} := (\psi - m_j)\circ q_{i,j}$ is convex on $q_{i,j}^{-1}(\Star(n_i))$, see Lemma~\ref{lem:pooh}.

\begin{rmk}\label{rmk:altcvx}
  Lemma~\ref{lem:pooh} gives an alternative proof a weaker version of Lemma~\ref{lem:cstar}, namely, that $\psi_{i,j}$ is convex on $q_{i,j}^{-1}(T_i^\circ)$ for every $\psi\in\cQ_{\sym}$. Indeed, the lemma implies that $\psi|_{T_i^\circ}$ is a supremum of functions of the form $\sum_{k\ne i}\theta_km_k+c$, with $c\in\R$ $\theta_k\ge0$, and $\sum_k\theta_k=1$. For each $j,k\ne i$, the function $(m_k+c-m_j)\circ q_{i,j}$ is affine, and this implies that $\psi_{i,j}$ is convex. In fact, even when $k=i$, the function $(m_k+c-m_j)\circ q_{i,j}$ is convex (although not affine); hence a similar argument can be used to prove the full statement of Lemma~\ref{lem:cstar}. 
\end{rmk}

\begin{lem}\label{lem:stuff}
Suppose that $\psi\in \mathcal{Q}_{\sym}$ and $i\ne j$. Then $p_{j,i}^{-1}$ gives a bijection of $(\partial^c\psi)(n)$ onto $\partial\psi_{i,j}(q_{i,j}^{-1}(n))$ for any $n\in T_i^{\circ}$. The same result holds for any $n\in \tau_j^{\circ}$. Moreover, the analogous results hold for $\cP_{\sym}$.
\end{lem}
\begin{proof}
  First suppose $n\in T_i^\circ$. By Lemma~\ref{lem:pooh}, we have $(\partial^c\psi)(n)\subset\sigma_i$. Lemma~\ref{lem:pairsym2} and symmetry of $\psi$ show that, for $m\in\sigma_i$, we have:
  \[
  \psi^c(m) = \sup_{n\in T_i} \langle m, n\rangle - \psi(n) =  \sup_{n\in \mathrm{Star}(n_i)} \langle m, n\rangle - \psi(n).
  \]
  Thus, $m\in (\partial^c\psi)(n)$ iff
  \[
    \langle m,n\rangle-\psi(n)
    \ge\langle m,n'\rangle-\psi(n')
  \]
  for all $n'\in \mathrm{Star}(n_i)$. \corr{Writing $m=p_{j,i}(x)$ and $n=q_{i,j}(y)$, Proposition~\ref{prop:chart_pair} implies that the above inequality is equivalent to 
  \[
    \langle x,y\rangle-\psi_{i,j}(x)
    \ge\langle x,y'\rangle-\psi_{i,j}(y')
  \]
  for all $y'\in q_{i,j}^{-1}(\Star(n_i))$, which amounts to $x\in\partial\psi_{i,j}(y)$.}
  The case when $n\in\tau_j^\circ$ is proved in the same way, using Lemma~\ref{lem:pooh}, and the proof for functions in $\cP_{\sym}$ is completely analogous.
  \begin{comment}
      Writing $m=p_{j,i}(s)$ and $n=q_{i,j}(t)$, Proposition~\ref{prop:chart_pair} implies that the above inequality is equivalent to 
  \[
    \langle s,t\rangle-\psi_{i,j}(t)
    \ge\langle s,t'\rangle-\psi_{i,j}(t')
  \]
  for all $t'\in q_{i,j}^{-1}(\Star(n_i))$, which amounts to $s\in\partial\psi_{i,j}(t)$.
  The case when $n\in\tau_j^\circ$ is proved in the same way, using Lemma~\ref{lem:pooh}, and the proof for functions in $\cP_{\sym}$ is completely analogous.
  \end{comment}
\end{proof}

Lemma~\ref{lem:stuff} allows us to apply many standard results for convex functions to $c$-convex functions.
For example, we have:
\begin{lem}\label{lem:cgrad3}
  If $\phi\in\cP_{\sym}$, then the following properties hold:
  \begin{itemize}
  \item[(i)]
    the c-subgradient $(\partial^c\phi)(m)$ is a singleton for almost every $m\in\Aspace$ (i.e.\ $\phi$ has a $c$-gradient a.e.);
  \item[(ii)]
    the a.e.\ defined function $(\partial^c\phi)\colon\Aspace\to\Bspace$ is measurable, and the set:
\[
\{n\in (\partial^c\phi)(m)\cap (\partial^c\phi)(m')\ |\ m,m'\in\Aspace_0,\ m \not = m'\}
\]
has Lebesgue measure 0.
  \end{itemize}
  Similar results hold for $\psi\in\cQ_{\sym}$.
\end{lem}
\begin{proof}
The convex function $\phi_{i,j}$ (defined analogously to $\psi_{i,j}$) is almost everywhere differentiable, so applying Lemma~\ref{lem:stuff} to each of the $S_i^{\circ}$, say, implies that $(\partial^c\phi)(m)$ is a singleton for a.e. $m\in\Aspace$, showing~(i).

The second point follows similarly, since the $q_{i,j}$ are measurable and $\Aspace$ is covered, up to a set of measure zero, by the sets $S_i^{\circ}$.
\end{proof}
Next we relate the c-transform on symmetric functions to the usual Legendre transform on $\R^d$. Denote by $L^\infty_{\sym}(\Aspace)$ and  $L^\infty_{\sym}(\Bspace)$ the sets of symmetric bounded functions on $\Aspace$ and $\Bspace$, respectively.
\begin{lem}\label{lem:c-transform-Leg-transform}
If $\p\in L^\infty_{\sym}(B)$ and $i\ne j$, then the convex function 
$(\psi^c-n_j)\circ p_{i,j}$ on $p_{i,j}^{-1}(S_i)\subset\R^d$ is the Legendre transform of the bounded function $\psi\circ q_{j,i}$ on $q_{j,i}^{-1}(\tau_i)\subset\R^d$. Similarly, the convex function $\psi^c\circ p_{i,j}$ on $p_{i,j}^{-1}(\sigma_j)$ is the Legendre transform of the convex function $\psi_{j,i}=(\psi-m_i)\circ q_{j,i}$ on $q_{j,i}^{-1}(T_j)$. The analogous statements hold for $\phi\in L^\infty(\Aspace)$.
\end{lem} 
\begin{proof}
  If $m\in S_i$, then Lemma~\ref{lem:pairsym2} implies that $\psi^c(m)=\sup_{n\in\tau_i}(\langle m,n\rangle-\psi(n))$. \corr{Writing  $m=p_{i,j}(x)$, $n=q_{j,i}(y)$, and using Proposition~\ref{prop:chart_pair}, we see that 
\begin{equation*}
  (\psi^c-n_j)(x)
  =\sup_{y\in q_{j,i}^{-1}(\tau_i)}(\langle x,y\rangle-\psi(p_{j,i}(y)),
\end{equation*}
which proves the first assertion. The remaining statements are proved in the same way.}
\end{proof}  
  
 Denote by $C^0_{\sym}(\Bspace)$ the set of symmetric continuous functions on $\Bspace$.
\begin{lem}\label{lem:candLeg}
  Suppose $\psi\in\cQ_{\sym}$ and $v\in C^0_{\sym}(\Bspace)$.
Then, for almost every $m\in\Aspace$, the function $t\mapsto(\psi+tv)^c(m)$ is differentiable at $t=0$, with derivative $-v((\partial^c\psi^c)(m))$.
\end{lem}
\begin{proof}
 Working in charts, using Lemma~\ref{lem:c-transform-Leg-transform}, this follows from the corresponding result about the Legendre transform on $\R^d$, as stated in~\eg\cite[Lemma~2.7]{BB13}.

A more direct proof goes as follows. Note that $(\psi+tv)^c(m)$ is convex in $t$. This means its left and right derivatives exist at $t=0$ and 
\begin{equation}
    \left.\frac{d(\psi+tv)^c(m)}{dt}\right|_{t=0^-} \leq \left.\frac{d(\psi+tv)^c(m)}{dt}\right|_{t=0^+}. 
    \label{eq:LeftRightDer}
\end{equation}
Assume $\partial^c\psi^c(m) = \{n\}$, and for each $t\not=0$, pick 
$n_t$ such that
\begin{equation} (\psi+tv)^c(m) = \langle m,n_t \rangle - \psi(n_t)-tv(n_t), \label{eq:PsitGrad}\end{equation}
which is possible by compactness of $B$.
By compactness of $B$, $\{n_t\}$ converges up to passing to subsequence to some $n_0\in B$ when $t\rightarrow 0$. 
We get, by continuity of the $c$-transform, that $n_0$ satisfies 
$\psi^c(m)+\psi(n_0) = \langle m,n_0 \rangle$; 
hence $n_0=n$ and $n_t\rightarrow n$. 
Using~\eqref{eq:PsitGrad}, this yields
\begin{eqnarray}  \frac{(\psi+tv)^c(m) - \psi^c(m)}{t} & = &  -v(n_t) - \frac{\psi(n_t) - \psi(n) - \langle m,n_t-n\rangle }{t}.
\end{eqnarray}
The numerator on the right hand side is positive since $m\in \partial^c\psi(n)$; hence 
$$ \left.\frac{d(\psi+tv)^c(m)}{dt}\right|_{t=0^+} \leq -v(n) \leq \left.\frac{d(\psi+tv)^c(m)}{dt}\right|_{t=0^-}. $$
Combining this with \eqref{eq:LeftRightDer} proves the lemma. 
\end{proof}

By Dominated Convergence, we obtain the following result, which is the analogue of the differentiability result needed to solve the complex and non-Archimedean Monge--Amp\`ere equations, respectively, see~\cite[Theorem~B]{BB10} and~\cite[\S7]{nama}. In what follows, $\mu$ denotes Lebesgue measure on $\Aspace$, of total mass $(d+2)^{d+1}/d!$.
\begin{cor}\label{cor:endiff}
  If $\psi\in\cQ_{\sym}$ and  $v\in C^0_{\sym}(\Bspace)$, then the function
  \begin{equation*}
    t\mapsto\int_{\Aspace}(\psi+tv)^c\,d\mu
  \end{equation*}    
  is differentiable at $t=0$, with derivative $-\int_{\Aspace} v((\partial^c\psi^c)(m))\,d\mu$.
\end{cor}

% 
%%%%%%%%%%%%%%%%%%%%%%%%%%%%%%%%%%%%%%%%%%%%%%%%%%%%%%%%%%%%%%%%%%%
%
\subsection{The tropical Monge--Amp\`ere measure}
We can now use Corollary~\ref{cor:endiff} to assign a symmetric positive measure $\nu_\psi$ on $\Bspace$ to any \emph{symmetric} function $\psi\in\cQ_{\sym}$, in a way which is compatible with the variational approach to the Alexandrov Monge--Amp\` ere operator. Hence, we will think of $\nu_\psi$ as the Monge--Amp\`ere measure of $\psi$.
\begin{defi}\label{defi:MA}
  Given $\psi\in\cQ_{\sym}$ we define a positive Radon measure $\nu_\psi$ on $\Bspace$ of mass $|\Aspace|$ by declaring 
  \begin{equation*}
    \int_{\Bspace}v\,d\nu_\psi
    :=-\frac{d}{dt}\bigg|_{t=0}\int_{\Aspace}(\psi+tv)^c\,d\mu
    = -\int_{\Aspace} (v \circ \partial^c\psi^c) \,d\mu
  \end{equation*}
  for every $v\in C^0_{\sym}(\Bspace)$.  
\end{defi}

\begin{prop}\label{prop:cMA}
For any $\psi\in\cQ_{\sym}$ and Lebesgue measurable $U\subset\Bspace$, we have
\[
\nu_{\psi}(U) = \mu((\partial^c\psi)(U)).
\]
\end{prop}
\begin{proof}
First, by Lemma~\ref{lem:cgrad3}, the multivalued map $\partial^c\psi^c$ is $\mu$-a.e. single valued, so by the standard change of variables formula and Corollary~\ref{cor:endiff}, we have:
\[
\nu_{\psi} = (\partial^c\psi^c)_*\mu.
\]
Since $\partial^c\psi^c$ and $\partial^c\psi$ are inverses, the result now follows from the definition of the pushforward measure.
\end{proof}

That $\nu_\psi$ is compatible with the Monge--Amp\`ere measure in charts now follows immediately from Lemma~\ref{lem:stuff}:
\begin{cor}\label{cor:MAincoords}
For any $\psi\in\cQ_{\sym}$ and any $i\ne j$, we have:
\begin{equation*}
  \nu_{\psi}|_{T_i^{\circ}}=(q_{i,j})_*\MAR\left(\psi_{i,j}|_{q_{i,j}^{-1}(T_i^{\circ})}\right)
  \quad\text{and}\quad
  \nu_{\psi}|_{\tau_j^{\circ}}=(q_{i,j})_*\MAR\left(\psi_{i,j}|_{q_{i,j}^{-1}(\tau_j^{\circ})}\right),
\end{equation*}
and for any $j$ we have
\begin{equation*}
  \nu_\psi|_{\tau_j^\circ}=\MAR((\psi\circ q_{i,j})|_{q_{i,j}^{-1}(\tau_j^\circ)}).
\end{equation*}
\end{cor}

Corollary~\ref{cor:MAincoords} allows us to now apply the standard theory for the Monge--Amp\`ere operator. For instance, we see that $\nu_\psi$ is weakly continuous under uniform convergence of the potentials, using the following result:
\begin{lem}\label{lem:convdir}
\cite[Theorem~2.1.22]{Hormander}  Let $\Omega\subset\R^d$ be an open convex subset, $u_j,u\colon\Omega\to\R$ convex functions, and assume that $u_j\to u$ pointwise on $\Omega$. Then there exists a subset $E\subset\Omega$ of full measure such that for every $x\in\Omega$, $u_j$ and $u$ are differentiable at $x$, and $u_j'(x)\to u'(x)$.
\end{lem}

\begin{prop}\label{prop:MAcont}
  If a sequence $(\psi_k)_{k=1}^\infty$ of functions in $\cQ_{\sym}$ converges uniformly to $\psi\in\cQ_{\sym}$, then $\nu_{\psi_k}\to\nu_\psi$ weakly as measures on $\Bspace$.
\end{prop}
\begin{proof}
  By definition, we have $\nu_\psi=(\partial^c\psi^c)_*\mu$, so by Dominated Convergence it suffices to prove that $\partial^c\psi_k^c\to\partial^c\psi^c$ a.e. Now, the c-transform is 1-Lipschitz, so we have $\psi_k^c\to\psi^c$ uniformly on $\Aspace$. Further, on the open stars $S_i^\circ$, which together have full measure, the c-gradient is computed as the gradient of a convex function on an open subset of $\R^d$, see Lemma~\ref{lem:stuff}. The result now follows from Lemma~\ref{lem:convdir}.
\end{proof}
\begin{rmk}\label{rmk:MAMetric}
As noted in~\S\ref{sec:AffRb}, any $\psi\in \cQ$ defines a convex metric $\Psi$ on an integral affine $\R$-bundle $\Lambda$ on the integral affine manifold $\Bspace_0$. Such a convex metric has a natural Monge--Amp\`ere measure $\MAR(\Psi)$, defined as $\MAR(\Psi\circ s)$ for any local affine section $s$, and Corollary~\ref{cor:MAincoords} shows that the restriction of $\nu_\psi$ to $\Bspace_0$ equals $\MAR(\Psi)$. Note, however, that $\nu_\psi$ may put mass also on $\Bspace\setminus\Bspace_0$, see Example~\ref{exam:singmass}.
\end{rmk}
% 
%%%%%%%%%%%%%%%%%%%%%%%%%%%%%%%%%%%%%%%%%%%%%%%%%%%%%%%%%%%%%%%%%%%
%
\subsection{Examples}\label{sec:examples}
We conclude by giving a few examples. Recall that the total mass of $\nu_\psi$ is always $\frac{(d+2)^{d+1}}{d!}$ for $\psi\in \mathcal{Q}_{\sym}$. For instance, when $d=2$, the total mass is $32$.  
\begin{exam}\label{exam:vertmass}
Let $\psi = \max_i m_i \equiv 1 \in \cQ_{\sym}$ and $m\in A$. The supremum 
$$ \sup_{n\in B} \langle m,n\rangle - \psi(n) = \sup_{n\in B} \langle m,n\rangle - 1 $$
is achieved at one of the vertices $\{n_0,\ldots,n_{d+1}\}$. It follows that $\partial^c\psi^c(m)$ contains a vertex for each $m\in B$. Consequently, since $\partial^c\psi^c$ is single valued almost everywhere, $\nu_\psi$ is supported at the vertexes and by symmetry $\nu_{\psi} = \sum_i \frac{(d+2)^d}{d!} \delta_{n_i}$.
\end{exam}
\begin{exam}\label{exam:dualvertmass}
For each $i$, let $n_i' := \frac{-n_i}{3}$ be the barycenters of the $\tau_i$. Using basic properties for the $c$-gradient, one can see that $\psi := (\max_i n_i')^c$ satisfies $\nu_\psi := \sum_i \frac{(d+2)^d}{d!}\delta_{n_i'}$. This can be computed explicitly, for example when $d=2$, 
$$\psi = \max \left\{\max_i m_i' -\frac{1}{9}, \max_{i\not=j} \frac{m_i + m_j}{2}-\frac{1}{3} \right\},$$ where $m_i' = \frac{-m_i}{3}$. %Example where $\nu_\psi$ pus mass only at the barycenters $n'_j$ of $\tau_j$, as discussed by M and N. 
\end{exam}
\begin{exam}\label{exam:singmass}
$\psi$ can also charge the singular set -- indeed, when $d=2$, one can verify that $\psi = \max\left\{\max_i m_i', \frac{1}{3} \right\}$ does not charge $B_0$ at all, and so we have $\nu_{\psi} = \sum_{x_i\in\Bspace\setminus\Bspace_0} \frac{16}{3}\delta_{x_i}$, by symmetry. One can also check that, while $\psi_{j,k}$ is actually convex on all of $q_{j,k}^{-1}(\Star(n_j)$, $\MAR(\psi_{j,k})(x_i) = \frac{80}{9} > \frac{16}{3}$ for each $x_i\in \mathrm{Star}(n_j)^{\circ}\setminus B_0$, so the equalities in Corollary~\ref{cor:MAincoords} cannot be extended to all of $\Star(n_j)^\circ$.
\end{exam}

For $\psi\in\cQ_{\sym}$, we have two equivalent definitions for $\nu_{\psi}$ (Definition~\ref{defi:MA} and Proposition~\ref{prop:cMA}), which agree with the Monge--Amp\`ere measure of $\psi_{i,j}$ in coordinates (Corollary~\ref{cor:MAincoords}). For non-symmetric $\psi\in\mathcal{Q}$, none of these are well-defined in general, and when they are, they need not agree with the Monge--Amp{\`e}re computed in coordinates, as the following examples show.
 
 \begin{exam}
    Let $d=2$ and $\psi := m_i$, for some fixed $i$; then the Monge--Amp\`ere measure of $\psi_{i,j}$ is $\MAR(\psi_{i,j}) = 128\delta_{0}$. 
    Since this gives a total mass larger than 32, we conclude that Corollary~\ref{cor:MAincoords} cannot hold for this $\psi$. 
 \end{exam}
 
 \begin{exam}
 Let $d=2$, and fix $0\leq i\leq 3$. If $\psi = \max_{j \not= i} m_j$, then one checks that $\mu(\partial^c\psi) = 8\delta_{n_i} + 32\delta_{n_i'}$; hence the right hand side in Proposition~\ref{prop:cMA} does not assign the correct total mass for non-symmetric $\psi$.    
 \end{exam}
 
 \begin{exam}\label{exam:non_diff}
 Let $d=1$, and $\psi(n) = \max_i \langle m_i, n - n_0'\rangle$, with $n_j' = \frac{-n_j}{2}$ for $j=0,1,2$. Let $v\geq 0$ be a piecewise linear function with $v(n_0) = 1$ and $v(n'_1)=v(n_2') = 0$. Then $(\psi + tv)^c(m)$ will not be differentiable in $t\in (-\e, \e)$ for all $m\in \sigma_0$; hence Definition~\ref{defi:MA} does not make sense for this $\psi$.
 \end{exam}

% 
%
%%%%%%%%%%%%%%%%%%%%%%%%%%%%%%%%%%%%%%%%%%%%%%%%%%%%%%%%%%%%%%%%%%%
%
%
\section{The tropical Monge--Amp\`ere equation}\label{sec:tropMA}
We are now ready to study the (symmetric) tropical Monge--Amp\`ere equation. Thus, given a symmetric positive measure $\nu$ on $\Bspace$ of mass $|A|$, we seek to find $\psi\in\cQ_{\sym}$ such that $\nu_\psi=\nu$. In particular, we will prove Theorem~B in the introduction.
%     
%%%%%%%%%%%%%%%%%%%%%%%%%%%%%%%%%%%%%%%%%%%%%%%%%%%%%%%%%%%%%%%%%%%
%
\subsection{Variational formulation}
Given a measure $\nu$ as above, we define a functional 
\[
  F=F_\nu\colon\cQ_{\sym}\to\R
\]
by
\[
  F(\psi):=\int_{\Aspace}\psi^c\,d\mu+\int_{\Bspace}\psi\,d\nu.
\]
\begin{lem}\label{lem:varappr}
  A function $\psi\in\cQ_{\sym}$ minimizes the functional $F$ iff $\nu_\psi=\nu$.
\end{lem}
\begin{proof}
  First suppose $\nu=\nu_\psi$. For any $\psi'\in\cQ_{\sym}$ we then have
  \begin{equation*}
    F(\psi')=\int_A(\psi^{\prime c}+\psi'\circ(\partial^c\psi^c))\,d\mu.
  \end{equation*}
  For almost every $m\in\Aspace$, we have
  \begin{equation*}
    \psi^{\prime c}(m)+\psi'((\partial^c\psi^c)(m))
    \ge\langle m,(\partial^c\psi^c)(m)\rangle
    =\psi^c(m)+\psi((\partial^c\psi^c)(m)),
  \end{equation*}
  and it follows that $F(\psi')\ge F(\psi)$, so that $\psi$ is a minimizer for $F$.

  Conversely, suppose that $\psi\in\cQ_{\sym}$ is a minimizer for $F$, and let us show that $\nu_\psi=\nu$.
  We must prove that $\int_{\Bspace}v\,d\nu_\psi=\int_{\Bspace}v\,d\nu$ for all $v\in C^0(\Bspace)$. As $\nu$ and $\nu_\psi$ are both symmetric, it suffices to establish this for $v\in C^0_{\sym}(\Bspace)$. Indeed, the function $\bar v=\frac1{|G|}\sum_{g\in G}v\circ g$ is symmetric, and we have  $\int_{\Bspace}v\,d\nu_\psi=\int_{\Bspace}\bar v\,d\nu_\psi$, 
$\int_{\Bspace}v\,d\nu=\int_{\Bspace}\bar v\,d\nu$.

Thus suppose $v\in C^0_{\sym}(\Bspace)$, and consider the function on $\R$ defined by
\[
  f(t):=\int_{\Aspace}(\psi+tv)^c\,d\mu+\int_{\Bspace}(\psi+tv)\,d\nu.
\]
 It follows from Corollary~\ref{cor:endiff} that $f(t)$ is differentiable at $t=0$, with
\begin{equation*}
  f'(0)=-\int_{\Bspace}v\,d\nu_\psi+\int_{\Bspace}v\,d\nu,
\end{equation*}
so we are done if we can prove that $f(t)$ has a (global) minimum at $t=0$.

Now $(\psi+tv)^{cc}\in\cQ_{\sym}$, and $(\psi+tv)^{ccc}=(\psi+tv)^c$, whereas $(\psi+tv)^{cc}\le\psi+tv$. Thus
  \begin{equation*}
    f(t)\ge\int_{\Aspace}(\psi+tv)^c\,d\mu+\int_{\Bspace}(\psi+tv)^{cc}\,d\nu
    =F((\psi+tv)^{cc})\ge F(\psi)=f(0),
  \end{equation*}
completing the proof.
\end{proof}
%     
%%%%%%%%%%%%%%%%%%%%%%%%%%%%%%%%%%%%%%%%%%%%%%%%%%%%%%%%%%%%%%%%%%%
%
\subsection{Existence and uniqueness}\label{sec:ExistenceUniqueness}
We will prove
\begin{thm}\label{thm:MA}
  For any symmetric positive measure $\nu$ on $\Bspace$ of total mass $|\Aspace|$, there exists a function $\psi\in\cQ_{\sym}$, such that $\nu_\psi=\nu$. Moreover, $\psi$ is unique up to an additive constant, and the map $\psi \mapsto \nu_{\psi}$ is a homeomorphism from $\mathcal{Q}_{\sym} / \R$ to the space $\cM_{\sym}(\Bspace)$ of positive, symmetric measures of mass $|A|$.
\end{thm}
\begin{proof}
  We use Lemma~\ref{lem:varappr}. To prove existence of a solution, it suffices to show that the functional $F_\nu$ admits a minimizer on $\cQ_{\sym}$. But as $\phi\mapsto\phi^c$ is Lipschitz continuous, one sees that $F$ is Lipschitz continuous. It is also translation invariant, so the existence of a minimizer follows from compactness of $\cQ_{\sym}/\R$, see Corollary~\ref{cor:compact}.

We now show uniqueness. It suffices to prove that if $\psi_0,\psi_1\in\cQ_{\sym}$ are two minimizers of $F$, normalized by $\int\psi_i\,d\nu=0$, $i=0,1$, then $\psi_0=\psi_1$. 
  Set $\psi:=\frac12(\psi_0+\psi_1)$. Then $\psi\in\cQ_{\sym}$, by convexity of $\cQ_{\sym}$, and we have $\int_\Bspace\psi\,d\nu=0$. Now
  \[
    \psi^c=\sup_{n\in\Bspace}(n-\psi(n))
    \le\frac12\sup_{n\in\Bspace}(n-\psi_0(n))+\frac12\sup_{n\in\Bspace}(n-\psi_1(n))
    =\frac12(\psi_0^c+\psi_1^c),
  \]
  pointwise on $\Aspace$. As $\int_\Bspace\psi\,d\nu=0$, this leads to 
  \[
    F_\nu(\psi)
    =\int_\Aspace\psi^c\,d\mu
    \le\frac12\int_\Aspace\psi_0^c\,d\mu+\frac12\int_\Aspace\psi_1^c\,d\mu
    =\min_{\cQ_{\sym}}F_\nu.
  \]
  Thus equality holds, so since $\psi_0^c,\psi_1^c,\psi^c$ are continuous, we must have $\psi^c=\frac12(\psi_0^c+\psi_1^c)$. For a.e.\ $m\in\Aspace$, $\psi_0^c$, $\psi_1^c$ and $\psi^c$ all admit a c-gradient at $m$. If we set $n:=(\partial^c\psi^c)(m)$, then 
  \begin{align*}
    \psi^c(m)
    &=\langle m,n\rangle-\psi(n)\\
    &=\frac12(\langle m,n\rangle-\psi_0(n))+\frac12(\langle m,n\rangle-\psi_1(n))\\
    &\le\frac12\psi_0^c(m)+\frac12\psi_1^c(m)\\
    &=\psi^c(m).
  \end{align*}
  Thus $\psi_i^c(m)=\langle m,n\rangle-\psi_i(n)$ for $i=1,2$, so we must have
  \[
    (\partial^c\psi_0^c)(m)=(\partial^c\psi_1^c)(m)=n
  \]
  for a.e.\ $m$.
  We may assume $m$ lies in some open star $S_i^\circ$. Pick any $j\ne i$. By Lemma~\ref{lem:stuff}, the convex functions $(\psi_0^c-m_j)\circ p_{j,i}$ and $(\psi_1^c-m_j)\circ p_{j,i}$ on $p_{j,i}^{-1}(S_i^\circ)$ have the same gradient at a.e.\ point. By Lemma~\ref{lem:convex} below, these two functions differ by an additive constant, so $\psi_0^c-\psi_1^c$ is constant on $S_i^\circ$. By continuity of the elements of $\cP$ and density of $\bigcup_i S_i^\circ$ in $\Aspace$, we get that $\psi_0^c-\psi_1^c$ is constant on $\Aspace$. It follows that $\psi_0-\psi_1$ is also constant, and hence zero, by our normalization.

  It follows that the Monge--Amp\`ere operator $\psi\mapsto\nu_\psi$ defines a bijection between the compact Hausdorff spaces $\cQ_{\sym}/\R$ and $\cM_{\sym}(\Bspace)$. By Proposition~\ref{prop:MAcont}, this bijection is continuous, and hence a homeomorphism.
\end{proof} 
\begin{lem}\label{lem:convex}
  If $\Omega\subset\R^n$ is open and convex, and $u_0,u_1$ are convex functions on $\Omega$ such that $\nabla u_0=\nabla u_1$ a.e.\ on $\Omega$, then $u_0-u_1$ is constant on $\Omega$.
\end{lem}
\begin{proof}
  Let $E\subset\Omega$ be the set of points where $\nabla u_0=\nabla u_1$. 
  Pick any point $x_0\in\Omega$. After adding a constant to $u_1$, we may assume $u_0(x_0)=u_1(x_0)$. Pick $r>0$ such that $B(x_0,2r)\subset\Omega$. It suffices to prove that $u_0=u_1$ on $B(x_0,r)$. Fubini's theorem implies that for almost every point $v$ on the unit sphere in $\R^n$, we have $x_0+tv\in E$ for almost every $t\in(-r,r)$. For such $v$ it follows that the convex functions $f_i(t):=u_i(x_0+tv)$ on $(-r,r)$ satisfy $f'_0(t)=f'_1(t)$ for a.e.\ $t$. As $f_0(0)=f_1(0)$, this implies that $f_0=f_1$, see~\cite[Theorem~3.35]{Folland}. Thus $u_0(x_0+tv)=u_1(x_0+tv)$ for almost every $v$ and all $t\in(-r,r)$. By continuity, we see that $u_0=u_1$ on $B(x_0,r)$.
\end{proof}

\begin{proof}[Proof of Theorem~B]
It is clear from Theorem~\ref{thm:MA} and Corollary~\ref{cor:MAincoords} that $\psi\mapsto\nu_\psi$ satisfies all the properties stated in Theorem~B. Now let $\psi\mapsto\nu'_\psi$ be a continuous map from $\cQ_{\sym}$ to $\cM_{\sym}$ such that $\nu'_\psi|_{\tau_i^\circ}=\MAR(\psi|_{\tau_i^\circ})$ for every $i$. Thus $\nu'_\psi=\nu_\psi$ for all $\psi\in\cQ_{\sym}$ such that $\nu_\psi$ puts full mass on $\bigcup_i\tau_i^\circ$. But it follows from Theorem~\ref{thm:MA} that the set of such functions is dense in $\cQ_{\sym}$; indeed, the set of measures on $\cM_{\sym}$ putting full mass on $\bigcup_i\tau_i^\circ$ is dense. The result follows.
\end{proof}

\begin{rmk}\label{rmk:regsol}
  It is of interest to the SYZ conjecture to investigate the regularity of the solution $\psi$ when $\nu$ is Lebesgue measure on $\Bspace$, using classical and more recent results, see~\cite{Figalli,Moo15,Moo21,MR22}. We hope to address this in future work.
\end{rmk}

% 
%
%%%%%%%%%%%%%%%%%%%%%%%%%%%%%%%%%%%%%%%%%%%%%%%%%%%%%%%%%%%%%%%%%%%
%
%
\section{Induced metrics on the Berkovich projective space}\label{sec:NAmetrics}
Here we define a procedure that to a symmetric c-convex function on $\Bspace$ associates a symmetric toric continuous psh metric on the Berkovich analytification of $\cO_{\P^{d+1}}(d+2)$.
As before, $K$ is any non-Archimedean field.
% 
%%%%%%%%%%%%%%%%%%%%%%%%%%%%%%%%%%%%%%%%%%%%%%%%%%%%%%%%%%%%%%%%%%%
%
\subsection{Continuous psh functions and metrics}
The study of continuous psh (or semipositive) metrics in non-Archimedean geometry goes back to Zhang~\cite{Zha95} and Gubler~\cite{GublerLocal}, who defined the notion of a continuous psh metric on the analytification of an ample line bundle on a projective variety defined over a nontrivially valued non-Archimedean field $K$.
This theory is global in nature.

More recently, a local theory was developed by Chambert-Loir and Ducros~\cite{CLD}. Given any non-Archimedean field $K$, any $K$-analytic space\footnote{All $K$-analytic spaces will be assumed good and boundaryless.} $Z$, can be endowed with a sheaf of continuous psh functions. For example, if 
$f_1,\dots,f_n$ are invertible analytic functions on $Z$, $\Omega\subset\R^n$ is an open subset such that $(\log|f_1(z)|,\dots,\log|f_n(z)|)\in\Omega$ for all $z\in Z$, and $\chi\colon\Omega\to\R^n$ is a convex function, then the function $z\mapsto\chi(\log|f_1(z)|,\dots,\log|f_n(z)|)$ is a continuous psh function on $Z$. A general continuous psh function is locally a uniform limit of functions of this type.

If $L$ is a line bundle (in the analytic sense) on $Z$, then a continuous metric on $L$ in the `multiplicative' sense, is a continuous function $\|\cdot\|$ on the total space of $L$ with values in $\R_{\ge0}$, and a suitable homogeneity property along the fibers of $L\to Z$. We say that $\|\cdot\|$ is semipositive if for some (equivalently, any) local analytic section $s\colon Z\to L$, the continuous function $-\log\|s\|$ on $Z$ is psh.
It will be natural for us to instead use `additive' terminology, and view a continuous metric on $L$ as an $\R$-valued function $\Psi=-\log\|\cdot\|$ on the total space with the zero section removed. If $\Psi$ is a continuous metric on $L$ and $s$ is a nonvanishing section of $L$ over an open set $U\subset Z^\an$,  then we can view $\Psi-\log|s|$ as a continuous function on $U$, and we say that $\Psi$ is psh if $\Psi-\log|s|$ is psh on $U$; this is equivalent to $\|\cdot\|=\exp(-\Psi)$ being semipositive.

In fact, the global notion in~\cite{GublerLocal,BE21} is a priori stronger. Let $X$ be a projective variety, and $L$ an ample line bundle on $L$. Then a continuous metric $\Psi$ on $L^\an$ is globally psh if it can be uniformly approximated by \emph{Fubini--Study metrics}, \ie metrics of the form
\begin{equation}\label{eq:FSMetrics}
  \Psi=\frac1r\max_{1\le j\le N}(\log|s_j|+\lambda_j),
  \end{equation}
where $r\ge 1$, $s_j\in\Hnot(X,rL)$ are global nonzero sections with no common zero, and $\lambda_j\in\R$. Such a metric is continuous psh in the sense above, and we shall only consider globally psh metrics.
% 
%%%%%%%%%%%%%%%%%%%%%%%%%%%%%%%%%%%%%%%%%%%%%%%%%%%%%%%%%%%%%%%%%%%
%
\subsection{Monge--Amp\`ere measures}
To any continuous psh function $\f$ on a pure-dimensional $K$-analytic space $Z$ is associated a `non-Archimedean' Monge--Amp\`ere measure $\MANA(\f)$, a positive Radon measure on $Z$. We refer to~\cite{CLD} for the definition, but note that the Monge--Amp\`ere operator is continuous under locally uniform convergence.

If $L$ is an analytic line bundle on $Z$ and $\Psi$ is a continuous psh metric on $L$, then the Monge--Amp\`ere measure $\MANA(\Psi)$ is a well-defined positive Radon measure on $Z$ with the following property:  for any nonvanishing section $s$ of $L$ over some open subset $U\subset Z$, we have $\MANA(\Psi)|_U=\MANA((\Psi-\log|s|)|_U)$. In `multiplicative' notation, this measure is written $c_1(L,\|\cdot\|)^d$, where $d=\dim Z$ and $\|\cdot\|=\exp(-\Psi)$, as first introduced by Chambert--Loir~\cite{CL06}.

In this paper, all computations involving the non-Archimedean Monge--Amp\`ere measure will be deduced from the following result, essentially due to Vilsmeier~\cite{Vilsmeier}.
\begin{lem}\label{lem:Vils}
  Let $T\simeq\Gm^n$ be a split torus, with tropicalization map $\trop\colon T^\an\to N_{T,\R}$. Let $\Omega\subset N_{T,\R}\simeq\R^n$ be an open subset, and $g\colon\Omega\to\R$ a convex function. Then the composition $g\circ\trop\colon\trop^{-1}(\Omega)\to\R$ is a continuous psh function, and we have
  \begin{equation*}
    \MA_{\NA}(g\circ\trop)=n!\MAR(g)
  \end{equation*}
  on $\trop^{-1}(\Omega)$, where the left-hand side denotes the non-Archimedean Monge--Amp\`ere measure on $\trop^{-1}(\Omega)$, and the right-hand side denotes the real Monge--Amp\`ere measure on $\Omega\subset\trop^{-1}(\Omega)$.
\end{lem}
\begin{proof}
  We argue as in the proof of~\cite[Corollary~5.10]{Vilsmeier}. 
  By ground field extension, we may assume $K$ is algebraically closed and non-trivially valued, and in particular has dense value group. We may also assume $T=\Spec K[z_1^\pm,\dots,z_n^\pm]$ and $N_{T,\R}=\R^n$. The statement is local on $\Omega\subset\R^n$, so \corr{pick any point $\underline{t}=(t_1,\dots,t_n)\in\Omega$, and nonzero elements $a,b_1,\dots,b_n\in K$ such that the set
  \[
    \{\underline{s}=(s_1,\dots,s_n)\in\R^n\mid s_j\ge\log|b_j|, \log|a|^{-1}+\sum_j\log|b_j|\ge \sum_js_j\}
    \]
    is contained in $\Omega$ and contains $\underline{t}$ in its interior.} After performing the change of coordinates $z_j\mapsto b_jz_j$ we may assume $b_j=1$ for all $j$. Now consider the formal scheme
    \[
      \mathfrak{X}=\operatorname{Spf}(K^\circ
      \langle z_0,\dots,z_n\rangle/(z_0\dots z_n-a)).
    \]
    The generic fiber of $\mathfrak{X}$ is isomorphic to the Laurent domain $|z_j|\le1$, $\prod_j|z_j|\ge|a|$ in $T^\an$, and the skeleton $\Delta$ of $\mathfrak{X}$ is the simplex $\{s_j\ge0, \sum s_j\le\log|a|^{-1}\}$. \corr{As this simplex contains the point $\underline{t}$ in its interior, the result now follows from Corollary~5.7 in~\cite{Vilsmeier}.}
\end{proof}
\begin{rmk}
  Lemma~\ref{lem:Vils} can also be deduced from~\cite{BJKM21} which systematically studies pluripotential theory for tropical toric varieties, and its relation to complex and non-Archimedean pluripotential theory.
\end{rmk}
% 
%%%%%%%%%%%%%%%%%%%%%%%%%%%%%%%%%%%%%%%%%%%%%%%%%%%%%%%%%%%%%%%%%%%
%
\subsection{The Fubini---Study operator}
We now return to the setup earlier in the paper.
There is a natural \emph{Fubini--Study operator} that associates to any continuous convex function $\phi\colon\Delta\to\R$ a continuous psh metric $\FS(\phi)$ on $\cO(d+2)^\an$, see for example Theorem~4.8.1 in~\cite{BPS}.\footnote{In~\cite{BPS}, the authors consider concave rather than convex functions on $\Delta$.}
It is characterized by the following two properties:
\begin{itemize}
\item
  $\phi\to\FS(\phi)$ is continuous;
\item
  if $\phi$ is $\Q$-PL, then for any sufficiently divisible $r\ge1$, we have 
  \[
    \FS(\phi)=\max_{m\in r\Delta\cap M}(r^{-1}\log|\chi^{r,m}|-\phi(r^{-1}m)).
  \]
\end{itemize}
Here $\phi$ is $\Q$-PL if it is the maximum of finitely many rational affine functions, \ie functions of the form $n+\lambda$, where $n\in N_\Q$ and $\lambda\in\Q$. Any continuous convex function on $\Delta$ is a uniform limit of $\Q$-PL convex functions, so the two conditions above completely determine the operator $\FS$.
%
%%%%%%%%%%%%%%%%%%%%%%%%%%%%%%%%%%%%%%%%%%%%%%%%%%%%%%%%%%%%%%%%%%%
%
\subsection{From c-convex functions to continuous psh metrics}
To any symmetric c-convex function $\psi\in\cQ_{\sym}$ we now associate a continuous psh metric on $\cO(d+2)^\an$. This metric, which we slightly abusively denote by $\FS(\psi)$, is defined by
\[
  \FS(\psi):=\FS(\psi^c|_\Delta),
\]
where we view the c-transform $\psi^c$ as a convex function on $M_\R$, defined using~\eqref{equ:ctransf2}. The map $\psi\mapsto\FS(\psi)$ is contractive, and equivariant for addition of constants.

The metric $\FS(\psi)$ is closely related to the canonical extension of $\psi$ to $N_\R$ in~\S\ref{sec:extend}. Indeed, if $m\in M\cap\Delta$, and $\chi^m := \chi^{1,m}\in\Hnot(\P^{d+1},\cO(d+2))$ is the corresponding section, then $\log|\chi^m|$ is a continuous metric on $\cO(d+2)^\an$ over $T^\an$. Then $\FS(\psi)-\log|\chi^m|$ is a continuous psh function on $T^\an$, and we have
\begin{equation}\label{equ:FStrop}
  \FS(\psi)-\log|\chi^m|=(\psi-m)\circ\trop,
\end{equation}
on $T^\an$, where $\trop\colon T^\an\to N_\R$ is the tropicalization map.

Lemma~\ref{lem:Vils} and~\eqref{equ:FStrop} allow us to compute the Monge--Amp\`ere measure of $\FS(\psi)$ on $T^\an$, but that is not what we want to do. Instead, we will consider the restriction of $\FS(\psi)$ to a Calabi--Yau hypersurface $X\subset\P^{d+1}$.
% 
%
%%%%%%%%%%%%%%%%%%%%%%%%%%%%%%%%%%%%%%%%%%%%%%%%%%%%%%%%%%%%%%%%%%%
%
%
\section{Calabi--Yau hypersurfaces}\label{sec:CY}
From now on, $K=\C\lau{\unipar}$. Consider a hypersurface $X\subset\P_K^{d+1}$ of the form
\[
X=V(z_0\cdot\ldots\cdot z_{d+1}+\unipar f(z)),
\]
where $f(z)\in\C[z]$ is a homogeneous polynomial of degree $d+2$ such that for any intersection $Z$ of coordinate hyperplanes $z_j=0$ in $\P^{d+1}$, $f$ does not vanish identically on $Z$ and $V(f|_Z)$ is smooth. We call such a polynomial \emph{admissible}. The set of admissible polynomials determines an open in $
|\mathcal{O}_{\mathbb{P}^{d+1}}(d+2)|$, which is not empty as, for instance, the Fermat polynomial  $f_{d+2}(z)=\sum_{i=0}^{d+1} z_i^{d+2}$ is admissible. Moreover, $X$ is smooth for any admissible polynomial.

%
%%%%%%%%%%%%%%%%%%%%%%%%%%%%%%%%%%%%%%%%%%%%%%%%%%%%%%%%%%%%%%%%%%%
%

\subsection{Models and skeletons}
Let $Y$ be any smooth and proper variety over $K$. Given a scheme $\cY$ over $R:= \C[[\unipar]]$, we denote by $\cY_K$, respectively $\cY_0$, the base change of $\cY$ to $K$, respectively to the residue field of $R$. A model of $Y$ is a flat $R$-scheme $\cY$ such that $\cY_K \simeq Y$. We say that the model is strict normal crossing (snc), respectively divisorially log terminal (dlt), if the pair $(\cY, \cY_{0,\redu})$ is so; see~\cite[Definitions 1.7, 1.18]{Kol13} for more details.

Given any snc, more generally dlt, model $\cY$ of $Y$, we denote by $\mathcal{D}(\cY_0)$ the dual complex of $\cY_0$ (see~\cite[Definition 3.62]{Kol13}); this admits a canonical embedding $\mathcal{D}(\cY_0) \xrightarrow{\simeq} \Sk(\cY) \subset Y^\val$ in $Y^\val$, whose image is called the skeleton of $\cY$.

\subsection{The essential skeleton of $X$}\label{sec:KSskel}
Let $X^\an$ be the analytification of $X$, and $X^\val\subset X^\an$ the set of valuations on the function field of $X$. We have that $X^\an$ is a closed subset of $\P^{d+1,\an}$. 
Being the analytification of a Calabi--Yau variety over $\C\lau{\unipar}$, $X^\an$ admits a canonical subset $\Sk(X)\subset X^\val$, the \emph{essential skeleton}, defined in two equivalent ways. On one side, $\Sk(X)$ is the locus where a certain weight function on $X^\an$ takes its minimal values,~\cite{KS06,MN15}; on the other side, $\Sk(X)$ is the skeleton associated with any minimal dlt model of $X$,~\cite{NX16}.

In our case, $\Sk(X)$ can be concretely described as follows.
Consider the model
\[
\cX:=V(z_0\cdot\ldots\cdot z_{d+1}+\unipar f(z))\subset\P^{d+1}_R
\]
of $X$ over $R=\C\cro{\unipar}$. The special fiber $\cX_0$ is simply given by $V(z_0\cdot\ldots\cdot z_{d+1})\subset\P^{d+1}_\C$, and its dual complex $\mathcal{D}(\cX_0)$ is evidently the boundary of a $(d+1)$-dimensional simplex, hence  topologically a sphere. We have an anticontinuous specialization map $\operatorname{sp}_\cX\colon X^\an\to\cX_0$.

Now $\cX$ is smooth away from 
$\Sing(\cX)= \bigcup_{i \neq j} V(z_i,z_j,\unipar, f(z))$. 
Since $f$ is admissible, 
\begin{itemize}
    \item[-] for any $\xi \in  V(z_{i_1},\ldots, z_{i_m},\unipar, f(z)) \subset \Sing(\cX)$ for some maximal $m \in \{2,\ldots,d\}$,
    \'{e}tale locally around $\xi$ we have the isomorphism $$\cX \simeq V(x_1\cdot \ldots\cdot x_m - \unipar y) \subset \mathbb{A}^{m+2}_{x_i,\unipar, y} \times \mathbb{A}^{d-m};$$
    \item[-] $(\cX, \cX_0)$ is snc at the generic point of each stratum of the special fiber $\cX_0$.
\end{itemize}

It follows that locally around any singular point, $\cX$ is a toric subvariety of $\mathbb{A}^{d+2}$ and the special fiber $\cX_0$ consists of toric divisors; by~\cite[Proposition 11.4.24]{CLS} the pair $(\cX, \cX_0)$ is log canonical around the singular points of $\cX$. Finally, one can check that $\cX$ is dlt by considering the small resolution of $\cX$ obtained as blow-up of all but one irreducible component of the special fiber (see~\cite[\S 2.1, 4.2]{Kol13}), and is minimal since $K_\cX \sim \mathcal{O}_X$. We conclude that $\Sk(X)= \Sk(\cX) \simeq \mathcal{D}(\cX_0)$.

Let us be a bit more precise, and describe the embeddings of the $d$-dimensional faces of $\mathcal{D}(\cX_0)$ in $X^\val$. Such a face is determined by a zero-dimensional stratum of the special fiber, say the point $\xi_i$, where $z_j=0$ for $j\ne i$. At $\xi_i$, $\cX$ is smooth, the rational functions
\[
w_{i,j}=\frac{z_j}{z_i},\ j\ne i,
\]
form a coordinate system at $\xi_i$, and $u_i:=-z_i^{d+2}/f(z)$ is a unit at $\xi_i$. We can write
\begin{equation}\label{equ:strcoord1}
z^{m_i}=\prod_{j\ne i}w_{i,j}^{-1}=\unipar^{-1}u_i
\end{equation}
and, for $j\ne i$,
\begin{equation}\label{equ:strcoord2}
z^{m_j} =w_{i,j}^{d+1} \cdot \prod_{l\ne i,j}w_{i,l}^{-1}.
\end{equation}

\begin{comment}
    Given numbers $t_j\in\R_{>0}$, $j\ne i$ with $\sum t_j=1$, there exists a unique minimal valuation $v_{i,t}$ on $\cO_{\cX,\xi_i}$ such that $v_{i,t}(w_{i,j})=t_j$ for all $i$. Then $v_{i,t}$ defines a point in $X^\val$. With a bit of work, one shows that $t\mapsto v_{i,t}$ extends to a homeomorphism of the closed simplex $\{\sum_{j\ne i}t_j=1\}\subset\R_{\ge0}^{\{0,\dots,d+1\}\setminus\{i\}}$ onto a compact subset $\tilde{\tau}_i$ of $X^\val$, Moreover, $\Sk(X)$ is a (non-disjoint) union of the $\tilde{\tau}_i$.

We equip each $\tilde{\tau}_i$ with an integral affine structure, in which the affine functions are integral linear combinations of $v\mapsto v(w_{i,j})$, $j\ne i$.

Let $U_i:=\operatorname{sp}_\cX^{-1}(\xi_i)\subset X^{\an}$ be the open subset of points specializing to the point $\xi_i$. Each point $v\in U_i$ lies in $T^\an$ as $f(z)$ is generic, and defines a semivaluation on $\cO_{\cX,\xi_i}$ such that $v(w_{i,j})>0$ for all $j\ne i$, and $\sum v(w_{i,j})=1$. There is a canonical retraction $r_i\colon U_i\to\tilde{\tau}_i^\circ$ defined by $r_i(v)=v_{i,t}$, where $t_j=v(w_{i,j})$.
\end{comment}

\corr{
Given numbers $\lambda_j\in\R_{>0}$, $j\ne i$ with $\sum \lambda_j=1$, there exists a unique minimal valuation $v_{i,\lambda}$ on $\cO_{\cX,\xi_i}$ such that $v_{i,\lambda}(w_{i,j})=\lambda_j$ for all $i$. Then $v_{i,\lambda}$ defines a point in $X^\val$. With a bit of work, one shows that $\lambda\mapsto v_{i,\lambda}$ extends to a homeomorphism of the closed simplex $\{\sum_{j\ne i}\lambda_j=1\}\subset\R_{\ge0}^{\{0,\dots,d+1\}\setminus\{i\}}$ onto a compact subset $\tilde{\tau}_i$ of $X^\val$, Moreover, $\Sk(X)$ is a (non-disjoint) union of the $\tilde{\tau}_i$.

We equip each $\tilde{\tau}_i$ with an integral affine structure, in which the affine functions are integral linear combinations of $v\mapsto v(w_{i,j})$, $j\ne i$.

Let $U_i:=\operatorname{sp}_\cX^{-1}(\xi_i)\subset X^{\an}$ be the open subset of points specializing to the point $\xi_i$. Each point $v\in U_i$ lies in $T^\an$ as $f(z)$ is generic, and defines a semivaluation on $\cO_{\cX,\xi_i}$ such that $v(w_{i,j})>0$ for all $j\ne i$, and $\sum v(w_{i,j})=1$. There is a canonical retraction $r_i\colon U_i\to\tilde{\tau}_i^\circ$ defined by $r_i(v)=v_{i,\lambda}$ where $\lambda_j=v(w_{i,j})$.
}
%
%%%%%%%%%%%%%%%%%%%%%%%%%%%%%%%%%%%%%%%%%%%%%%%%%%%%%%%%%%%%%%%%%%%
%
\subsection{Tropicalization} \label{subsec:trop}
The complement $\P^{d+1,\an}\setminus T^\an$ consists of the hyperplanes $z_i=0$ and do not meet $X^\val$, so $X^\val\subset T^\an$. Note, however, $X^\val\cap T^\val=\emptyset$.

By definition, the tropicalization of $X\cap T$ (viewed as a subset of $T$ and with respect to the given valuation on the ground field), is the image
\[
(X\cap T)^{\trop}:=\trop(X^\an\cap T^\an)\subset N_\R.
\]
By the Fundamental Theorem of Tropical Geometry (FTTG for short, due to Kapranov in the case of hypersurfaces~\cite[Theorem~3.1.3]{MS15}), the tropicalization admits a different description. Namely, the Laurent polynomial $g(z)=1+\unipar\frac{f(z)}{z_0\cdot\ldots\cdot z_{d+1}}\in K[M]$ can be written as
\[
g(z)=1+\unipar\sum_{m\in\Delta}a_mz^m,
\]
where $a_m\in\C$ and $a_m\ne0$ whenever $m$ is a vertex of $\Delta$. Then we have $X\cap T=V(g)$. The tropicalization of $g$ is the convex, piecewise affine function on $N_\R$ given by 
\[
g^{\trop}(n)
=\max\{0,-1+\max_{m\in\Delta\cap M, a_m\ne0}\langle m,n\rangle\}
=\max\{0,-1+\max_{0\le i\le d+1}\langle m_i,n\rangle\},
\]
and the FTTG says that $(X\cap T)^{\trop}$ is the locus where the function $g^{\trop}$ fails to be locally affine. Its complement in $N_\R$ has a unique bounded component, namely the set
\[
\{n\in N_\R\mid \max_i\langle m_i,n\rangle<1\},
\]
whose closure is exactly the simplex $\Delta^\vee$ above. In particular,  $\Bspace=\partial\Delta^\vee\subset(X\cap T)^{\trop}$.
\begin{lem} \label{lem:tropX}
	The map $\trop\colon X^\an\cap T^\an\to(X\cap T)^{\trop}\subset N_\R$ 
	\begin{enumerate}
		\item induces a homeomorphism of $\tilde{\tau}_i$ onto $\tau_i$, and of $\Sk(X)$ onto $\Bspace$;
		\item fits in the commutative diagram \quad
		$\begin{tikzcd}[row sep=tiny]
		& {\tau_i^\circ} \\
		{U_i} \\
		& {\tilde{\tau_i}^\circ}
		\arrow["\trop", from=2-1, to=1-2]
		\arrow["{r_i}"', from=2-1, to=3-2]
		\arrow["\simeq","\trop"', from=3-2, to=1-2]
		\end{tikzcd}$
		\item in particular, satisfies  $U_i 
		:=\operatorname{sp}_\cX^{-1}(\xi_i)
		=\trop^{-1}(\tau_i^\circ)$;
		\item induces an isomorphism between the integral affine structures on $\tilde{\tau}_i$ and $\tau_i$.
	\end{enumerate}
\end{lem}

\begin{proof}
\corr{
We consider a $d$-dimensional simplex $\tilde{\tau}_i$ in $\Sk(X)$ as described in~\S\ref{sec:KSskel}.
Given numbers $\lambda_j\in\R_{\ge0}$, $j\ne i$ with $\sum \lambda_j=1$, let $v_{i,\lambda}$ be the minimal  valuation on $\cO_{\cX,\xi_i}$ such that $v_{i,\lambda}(w_{i,j})=\lambda_j$ for all $j\ne i$. By definition, $\trop(v_{i,\lambda})\in N_\R$ satisfies
	\[
	\langle m,\trop(v_{i,\lambda})\rangle=-v_{i,\lambda}(z^m)
	\]
	for all $m\in M$.
From~\eqref{equ:strcoord1} and~\eqref{equ:strcoord2} we have
	$
	\langle m_i,\trop(v_{i,\lambda})\rangle=-v_{i,\lambda}(\unipar^{-1}u_i)=1
	$
	and
	\begin{align} \label{eq:tropm_j}
	\begin{split}
	\langle m_j,\trop(v_{i,\lambda})\rangle
	&=-(d+1)v_{i,\lambda}(w_{i,j})+ \sum_{l \neq i,j}v_{i,\lambda}(w_{i,l})\\
	&=-(d+1)\lambda_j +\sum_{l \neq i,j}\lambda_l =- (d+2)\lambda_j +1 \leq 1,
	\end{split}
	\end{align}
	for $j\ne i$. This means that $\trop(v_{i,\lambda})$ lies in $\tau_i \subset \Bspace$, thus $\trop(\tilde{\tau}_i) \subseteq \tau_i$. The composition }
\[\begin{tikzcd}[row sep=small]
	\{\sum_{j \neq i} \lambda_j=1\}
	& \tilde{\tau}_i
	& {\tau_i}\\
	{\mathbb{R}_{\geq 0}^{\{0,\ldots,d+1\} \setminus \{i\}}} 
	& {\Sk(X)} 
	& B
	\arrow[hook, from=1-1, to=2-1]
	\arrow["\simeq", from=1-1, to=1-2]
	\arrow["\trop", from=1-2, to=1-3]
	\arrow[hook, from=1-2, to=2-2]
	\arrow[hook, from=1-3, to=2-3]
\end{tikzcd}\]
is injective by \eqref{eq:tropm_j} and surjective. \corr{Indeed, given $n \in \tau_i$, set $$\lambda_j = \frac{1}{d+2}(1-\langle n,m_j \rangle)$$
	for $j \neq i$. As $-(d+1) \leq \langle n,m_j \rangle \leq 1$ we have $\lambda_j \in \mathbb{R}_{\geq 0}$. Moreover,
	$$ \sum_{j \neq i} \lambda_j 
	= \frac{1}{d+2}(d+1 - \langle n, \sum_{j \neq i }m_j \rangle)
	=\frac{1}{d+2}(d+1 + \langle n, m_i \rangle)= 1.$$ }
	It follows that the restriction $\trop\colon \tau_i \rightarrow \tilde{\tau}_i$ is a homeomorphism, and $\trop\colon \Sk(X) \rightarrow B$ does too. 
	
	Part (3) follows directly from (2). For (2), we first check that $\trop(v) \in \tau_i^\circ=\{\max_{j \neq i}m_j<m_i=1\} $ for any $v \in U_i$. Indeed, we have $
	\langle m_i,\trop(v)\rangle=-v(\unipar^{-1}u_i)=1
	$ and $
	\langle m_j,\trop(v)\rangle=1-(d+2)v(w_{i,j})<1
	$ since $z_j=0$ at $\xi_i$.
	To conclude, it is enough to check that $\trop(v)=\trop \circ r_i(v)$ on an integral basis for $M$: \corr{
	\[
	\langle \trop(v), e_j-e_i \rangle 
	= -v(w_{i,j})=-\lambda_j 
	= -v_{i,\lambda}(w_{i,j}) 
	= \langle \trop(v_{i,\lambda}), e_j-e_i \rangle.
	\]}
	
	For (4), we recall that the affine functions on $\tilde{\tau}_i$ are integral linear combinations of $v \mapsto v(w_{i,j})$, for $j \neq i$. The affine functions on $\tau_i$ are the elements of $M$; as the set $\{e_j-e_i\}_{j \neq i}$ forms an integral basis for $M$, the affine functions on $\tau_i$ are integral linear combinations of $n \mapsto \langle e_j - e_i, n \rangle$. \corr{From 
	$
	\langle \trop(v_{i,\lambda}), e_j-e_i \rangle =-v_{i,\lambda}(w_{i,j})
	$
	we obtain (4).
}
\end{proof}

\subsection{Affinoid torus fibration}\label{sec:afftor}
As explored in~\cite{KS06,NXY19,MP21}, the analytification $X^\an$ of the Calabi--Yau variety $X$ admits various affinoid torus fibrations. 
Here we show that the tropicalization map induces affinoid torus fibrations on the open subsets $U_i$ of $X^\an$.

We recall that a continuous map $\rho: Y^\an \to S$ to a topological space $S$ is an $n$-dimensional affinoid torus fibration at a point $s \in S$ if there exists an open neighborhood $U$ of $s$
in $S$, such that the restriction to $\rho^{-1}(U)$ fits into a commutative diagram:

\[\begin{tikzcd}
\rho^{-1}(U)	
& \trop^{-1}(V)\\
U 
& V,
\arrow["\simeq", from=1-1, to=1-2]
\arrow["\rho"', from=1-1, to=2-1]
\arrow["\simeq", from=2-1, to=2-2]
\arrow["\trop", from=1-2, to=2-2]
\end{tikzcd}\]
$V$ being an open subset of $\mathbb{R}^n$, the upper horizontal map an isomorphism of analytic spaces, the
lower horizontal map a homeomorphism, and the map $\trop$ defined as in~\S\ref{subsec:trop map} for a torus of dimension $n$. In particular, an affinoid torus fibration induces an integral affine structure on the base $S$; see~\cite[\S 4.1]{KS06} for more details.

\begin{cor}\label{cor:afftor_opentop}
For any $i$, the map $\trop:U_i \to \tau_i^\circ$ is an affinoid torus fibration. Moreover, the induced integral affine structure on $\tau_i^\circ$ agrees with the one constructed in~\S\ref{sec:SIAS}.
\end{cor}

\begin{proof}
By Lemma~\ref{lem:tropX}, $\trop|_{U_i}$ is homeomorphic to the retraction $r_i$ that is an affinoid torus fibration, see for instance the proof of~\cite[Theorem 6.1]{NXY19}. Again by Lemma~\ref{lem:tropX}, $\tilde{\tau}_i^\circ$ and $\tau_i^\circ$ are isomorphic as integral affine manifolds, hence the integral affine structure induced on $\tau_i^\circ$ by the affinoid torus fibration $\trop$ is isomorphic to the one constructed in~\S\ref{sec:SIAS}.
\end{proof}

%
% 
%%%%%%%%%%%%%%%%%%%%%%%%%%%%%%%%%%%%%%%%%%%%%%%%%%%%%%%%%%%%%%%%%%%
%
%
\section{Solution to the non-Archimedean Monge--Amp\`ere equation}\label{sec:NAMA}
Let $L:=\cO(d+2)|_X$. We are now ready to prove Theorem~A in the introduction. Namely, we show that the preceding method recovers the solution to the non-Archimedean Monge--Amp\`ere equation~\cite{nama}, for a symmetric measure supported on the skeleton.
%
%%%%%%%%%%%%%%%%%%%%%%%%%%%%%%%%%%%%%%%%%%%%%%%%%%%%%%%%%%%%%%%%%%%
%
\subsection{Existence and uniqueness}
As mentioned above, to any (global) continuous psh metric $\Psi$ on $L^\an$ is associated a measure $\MANA(\Psi)$ on $X^\an$. By~\cite{YZ17} we have  $\MANA(\Psi_1)=\MANA(\Psi_2)$ iff $\Psi_1-\Psi_2$ is a constant, whereas the main result of~\cite{nama} (see also~\cite{BGJKM20}) states that for any measure $\nu$ supported on $\Sk(X)\subset X^\an$, there exists $\Psi$ such that $\MANA(\Psi)=\nu$.
% 
%%%%%%%%%%%%%%%%%%%%%%%%%%%%%%%%%%%%%%%%%%%%%%%%%%%%%%%%%%%%%%%%%%%
%
\subsection{Comparing Monge--Amp\`ere measures}
Given $\psi\in\cQ_{\sym}$ we want to compare the tropical Monge--Amp\`ere measure $\nu_\psi$ on $\Bspace$ with the non-Archimedean Monge--Amp\`ere measure of the continuous psh metric $\FS(\psi)$ on $L^\an$, 
via the embedding $\Bspace\simto\Sk(X)\subset X^\an$.

\begin{thm}\label{thm:CompareMA}
  Let $\psi\in\cQ_{\sym}$ be any symmetric c-convex function.
  Then the associated continuous psh metric $\FS(\psi)$ on $L^\an$ has Monge--Amp\`ere measure $\MANA(\FS(\psi))=d!\,\nu_\psi$, viewed as a measures on $\Bspace\simeq\Sk(X)\subset X^\an$.
\end{thm}
To prove the theorem, we want to use Vilsmeier's result in Lemma~\ref{lem:Vils}, but our torus $T$ is of the wrong dimension. Instead, we use the fact that $X^\an$ admits local affinoid torus fibrations with bases that are open subsets of $\Bspace\simeq\Sk(X)$, and that these fibrations are compatible with the embedding of $X$ into the toric variety $\P^{d+1}$, see~\S\ref{sec:afftor}.

\begin{proof}
  We first consider the case when the measure $\nu_\psi$ gives full mass to the union of the open $d$-dimensional simplices $\tau_i^\circ$ of $\Bspace$. By Corollary~\ref{cor:afftor_opentop}, the tropicalization map gives a affinoid torus fibration $\trop\colon U_i\to\tau_i^\circ$. For any $i\ne j$ we have
  \[
    \FS(\psi)-\log|\chi^{m_j}|=(\psi-m_j)\circ\trop
  \]
  on $U_i$, by~\eqref{equ:FStrop}.  Lemma~\ref{lem:Vils} and Corollary~\ref{cor:MAincoords} now give that $\MANA(\FS(\psi))=d!\,\nu_\psi$ on $\tau_i^\circ\subset U_i$. As $\MANA(\FS(\p))$ and $\nu_\psi$ have mass $d!|\Aspace|$ and $|\Aspace|$, respectively, we have accounted for all the mass of $\MANA(\FS(\p))$, so $\MANA(\FS(\p))=d!\nu_\psi$.
  
  \smallskip
  Now consider the general case. We can find a sequence $(\nu_n)$ of symmetric measures on $\Bspace$ of mass $|\Aspace|$ converging weakly to $\nu_\psi$ and such that $\nu_n$ gives full mass to $\bigcup_i\tau_i^\circ$. By what precedes, there exists a unique $\psi_n\in\cQ_{\sym}$ such that $\nu_{\psi_n}=\nu_n$ and $\int_\Bspace\psi_n\,d\nu=0$. By compactness of $\cQ_{\sym}/\R$ we may assume that $\psi_n$ converges uniformly to a function $\psi\in\cQ_{\sym}$. By continuity of the Fubini--Study operator, $\FS(\psi_n)$ converges uniformly to $\FS(\psi)$. It follows that the Monge--Amp\`ere measures $\MANA(\FS(\psi_n))$ converge weakly to $\MANA(\FS(\psi))$. Since  $\MANA(\FS(\psi_n))=d!\,\nu_n$ for all $n$, we get $\MANA(\FS(\psi))=d!\,\nu_\psi$, and we are done.
\end{proof}
%
%%%%%%%%%%%%%%%%%%%%%%%%%%%%%%%%%%%%%%%%%%%%%%%%%%%%%%%%%%%%%%%%%%%
%
\subsection{Invariance under retraction}\label{sec:invretr}
Let $\nu$ be a symmetric positive measure of mass $(d+2)^{d+1}$ on $\Sk(X)$. 
The results above give a rather explicit description of the solution (which is unique, up to a constant) to the non-Archimedean Monge--Amp\`ere equation $\MANA(\psi)=\nu$ on the Calabi--Yau variety $X\subset\P^{d+1}$. For one thing, $\psi$ is the restriction of a torus invariant metric on $\cO_{\P^{d+1}}(d+2)$. Note that
we are \emph{not} assuming that $X$ is invariant under any torus action.

Here we investigate further properties of the solution.

\begin{cor}\label{cor:retrinv}
  Let $\nu$ be any symmetric positive measure on $\Sk(X)$ of mass $(d+2)^{d+1}$, and let $\Psi$ be a continuous psh metric on $L^\an$, 
  whose Monge--Amp\`ere measure equals $\nu$. Then $\Psi$ is invariant under retraction to $\Sk(X)$, in the following sense: for any $j\ne i$, the function $\Psi_{i,j}:=(\Psi-\log|\chi^{m_j}|)|_{U_i}$ satisfies $\Psi_{i,j}=\Psi_{i,j}\circ r_i$.

\end{cor}
\begin{proof}
By Theorem~\ref{thm:MA} there exists a function $\psi\in\cQ_{\sym}$ such that $\nu_\psi=\nu$, and by Theorem~\ref{thm:CompareMA} $\Psi=\FS(\psi)$. By \eqref{equ:FStrop} we have $\FS(\psi)-\log|\chi^{m_j}|=(\psi-m_j)\circ\trop$ on $T^\an$.
By Lemma~\ref{lem:tropX}, we have $\trop= \trop \circ r_i$ on $U_i$, hence the claim follows.
\end{proof}
%
% 
%%%%%%%%%%%%%%%%%%%%%%%%%%%%%%%%%%%%%%%%%%%%%%%%%%%%%%%%%%%%%%%%%%%
%
%
\section{Applications to the SYZ conjecture}\label{sec:SYZ}
In this section we prove Corollary~C, following the work of Li. As $f$ is admissible, for any $\unipar\in\C^*$, \begin{equation*}
  X_\unipar:=\{z_0\dots z_{d+1}+\unipar f_{d+2}(z)=0\}\subset\P^{d+1}_\C
\end{equation*}
is a smooth complex projective variety, which we view as a complex manifold. Set
\[ \alpha_\unipar:=\tfrac1{(\log|\unipar|^{-1})}c_1(\cO_{\P^{d+1}}(d+2)|_{X_\unipar}).\]
We equip $X_\unipar$ with the unique Ricci flat K\"ahler metric in $\alpha_\unipar$. Let $\nu_\unipar$ be the corresponding smooth positive measure on $X_\unipar$, and write $(X_\unipar,d_\unipar)$ for associated  metric space.
% 
%%%%%%%%%%%%%%%%%%%%%%%%%%%%%%%%%%%%%%%%%%%%%%%%%%%%%%%%%%%%%%%%%%%
%
\subsection{Proof of Corollary~C}\label{sec:pfCorC}
We follow~\cite{LiSYZ}. 
Let us identify 
$$\cX=\{z_0\cdots z_{d+1}+\unipar f_{d+2}(z)=0\}$$ with the associated (singular) $d+1$-dimensional complex analytic subspace of $\P^{d+1}\times\C$. The central fiber $\cX_0$ consists of the $d+2$ coordinate hyperplanes $H_i=\{z_i=0\}$ in $\P^{d+1}\simeq\P^{d+1}\times\{0\}$. For each $i$, let $\xi:=\bigcap_{j\ne i}H_j\subset\cX_0$. Then $\cX$ is smooth at $\xi_i$, and we may find local holomorphic coordinates $w_j$, $j\ne i$, at $\xi_i$ such that $\unipar=w_0\cdot \ldots \cdot w_d$. If we pick small (disjoint) neighborhoods $W_i$ of $\xi_i$, and set $W:=\bigcup_iW_i$, then we have a continuous map
\[ \Log_\cX\colon W\setminus\cX_0\to\bigcup_i\tau_i^\circ\subset\Bspace \]
defined by $\Log_\cX=\sum_{j\ne i}\frac{\log|w_j|}{\log|\unipar|}n_j$ on $W_i$.

By~\cite{konsoib} (see~\cite[\S3.1]{LiSYZ}), most of the mass of $X_\unipar$ lies in $W$ for $\unipar\approx0$. Indeed:
\begin{lem}\label{lem:muchmass}
We have $\lim_{\unipar\to0}\nu_\unipar(X_\unipar\cap W)/\nu_\unipar(X_\unipar)=1$.
\end{lem}
\begin{proof}
Viewed as a scheme over $\C\cro{\unipar}$, $\cX$ is a minimal dlt model, whose skeleton $\Sk(\cX)$ equals the essential skeleton $\Sk(X)$. 
If $\cX$ were a semistable model, the lemma would follow from~\cite{konsoib}. Now, there exists a projective birational morphism $\pi\colon\cX'\to\cX$ such that $\cX'$ is an snc model of $X$, and such that $\pi$ is an isomorphism over the regular part of $\cX$. It follows that the only $d$-dimensional simplices of $\Sk(\cX')$ contained in $\Sk(X)$ are associated to $\xi'_i:=\pi^{-1}(\xi_i)$, $0\le i\le d$. We can pick corresponding neighborhoods $W'_i$ of $\xi'_i$ such that $\pi(W'_i)\subset W_i$ for all $i$, and hence $X_\unipar\cap W'\subset X_\unipar$, where $W'=\bigcup_iW'_i$. By~\cite[Theorem~3.4]{konsoib}, we have $\lim_{\unipar\to0}\nu_\unipar(X_\unipar\cap W')/\nu_\unipar(X_\unipar)=1$, and the result follows. 
\end{proof}

Let $\psi\in\cQ_{\sym}$ be such that $\nu_\psi$ equals Lebesgue measure on $\Bspace\simeq\Sk(X)$. In particular, the restriction of $\psi$ to a $d$-dimensional face $\tau_i^\circ$ satisfies the real Monge--Amp\`ere equation, and is therefore smooth outside a small closed subset.

Let $\phi:=\psi^c$ be the c-transform of $\psi$, viewed as a continuous convex function on $\Delta$. This defines a continuous psh metric on $\cO_{\P^{d+1}}(d+2)^{\an}$ whose restriction to $X^\an$ has Monge--Amp\`ere measure equal to $\nu$, by Theorem~\ref{thm:CompareMA}.

The continuous convex function on $\Delta$ also defines a continuous psh metric on the holomorphic line bundle $\cO(d+2)$ on $\P^{d+1}$. Approximating $\psi$ by a smooth strictly convex function, we can approximate this metric uniformly by a K\"ahler metric on $\cO(d+2)$. As in~\cite[Lemma~4.1]{LiSYZ}, this leads to the existence, given $\e>0$, of a K\"ahler metric $\om_\unipar$ on $(X_\unipar,\alpha_\unipar)$, such that, for any $i$, $\om_\unipar$ has a local potential on $W_i\cap X_\unipar$ that differs from the function $\psi\circ\Log_\cX$ by at most $\e$.

As explained in~\cite[Lemma~4.2]{LiSYZ}, we may, for small $\unipar$ and $\epsilon$, by shrinking $W$ (so that $\Log_\cX(W\cap X_\unipar)$ is contained in the smooth locus of $\psi$) find Lipschitz functions $f_\unipar$ of $C^0$-norm on the order of $\epsilon$, smooth outside a set of measure 0, such that the $(1,1)$-currents $\omega_{f,\unipar} = \omega_\unipar+dd^cf_\unipar$  is positive on $X_\unipar$ and approximate $dd^c(\psi\circ\Log_\cX)$ on the smooth locus of $f$ in $W$ and the measure $\omega_{f,\unipar}^d$ is close to the Calabi--Yau volume form $\nu_\unipar$ on $X_\unipar$ in total variation. The proof then proceeds to get the $C^0$ convergence on $W$ of the potentials of the Calabi--Yau metrics on $X_\unipar$ to $\psi\circ \Log_\cX$ and the resulting $C^\infty$-convergence and special Lagrangian torus fibration on $W$ in the same way as in~\S4.3 and~\S4.5 of~\cite{LiSYZ}.

% 
%%%%%%%%%%%%%%%%%%%%%%%%%%%%%%%%%%%%%%%%%%%%%%%%%%%%%%%%%%%%%%%%%%%
%
\subsection{Gromov--Hausdorff convergence}
Consider the Fermat family
\begin{equation*}
X_{\unipar} = \{ z_0z_1\ldots z_{d+1} + \unipar(z_0^{d+2}+\dots+z_{d+1}^{d+2}) \} \subset \P_\C^{d+1}.
\end{equation*}
As above, we write $(X_\unipar,d_\unipar)$ for the corresponding metric space.

Let $\psi\in\cQ_{\sym}$ be a solution (unique, up to a constant) to the tropical Monge--Amp\`ere equation $\nu_\psi=\nu$, where $\nu$ is Lebesgue measure on $\Bspace$, see Theorem~\ref{thm:MA}, and let $\Psi$ be the corresponding metric on the affine $\R$-bundle $\Lambda$, see Remark~\ref{rmk:MAMetric}. By the regularity theory for the real Monge--Amp\`ere equation on $\R^d$, there exists an open subset $\cR_\psi\subset\Bspace_0$ such that $\Bspace\setminus\cR_\psi$ has $(d-1)$-Hausdorff measure zero, and such that $\Psi$ is smooth and strictly convex over $\cR_\psi$. The Hessian of $\Psi$ on $\cR_\psi$ then defines a metric on $\cR$, and we let 
$(\cR_\psi,d_\psi)$ be the resulting metric space. 

By diameter bounds proved in \cite{LiTosatti}, $(X_\unipar,d_\unipar)$ converges in the sense of Gromov--Hausdorff after passing to subsequence. By \cite[Theorem~5.1]{LiFermat}, any subsequential limit $(X_\unipar,d_\unipar)$ contain a dense subset locally isomorphic to the regular part of a Monge--Amp{\`e}re metric on $B_0$. By the injectivity in Theorem~B, the latter space is uniquely determined as $(\cR_\psi,d_\psi)$.
% 
% 
%%%%%%%%%%%%%%%%%%%%%%%%%%%%%%%%%%%%%%%%%%%%
%
%

%
% 
%
%%%%%%%%%%%%%%%%%%%%%%%%%%%%%%%%%%%%%%%%%%%%%%%%%%%%%%%%%%%%%%%%%%%
%
%
%

\begin{thebibliography}{MMRZ22}

\bibitem[AG13]{AG13} 
L.~Ambrosio and N.~Gigli.
\newblock \emph{A user's guide to optimal transport}.
\newblock In \emph{Modelling and Optimisation of Flows on Networks}.
\newblock Lecture Notes in Mathematics, 2062.
%, DOI 10.1007/978-3-642-32160-3 1,
\newblock Springer-Verlag Berlin Heidelberg 2013.

\bibitem[AH23]{AH}
R.~Andreasson. J.~Hultgren.
\newblock \emph{Solvability of Monge-Amp\`ere equations and tropical affine structures on reflexive polytopes }.
\newblock \texttt{arXiv:2303.05276}.

\bibitem[BB13]{BB13} 
R.~J.~Berman and B.~Berndtsson.
\newblock \emph{Real Monge--Amp\`ere equations and K\"ahler--Ricci solitons on toric log Fano varieties}.
\newblock Ann. Fac. Sci. Toulouse \textbf{22} (2013), 649--711.

\bibitem[BB10]{BB10} 
R.~J.~Berman and B.~Boucksom.
\newblock \emph{Growth of balls of holomorphic sections and energy at equilibrium}.
\newblock Invent. Math. \textbf{181} (2010), 337--394.

\bibitem[BBGZ13]{BBGZ}
R.~J.~Berman, S.~Boucksom, V.~Guedj and A.~Zeriahi.
\newblock \emph{A variational approach to complex Monge--Amp\`ere equations}.
\newblock Publ. Math. Inst. Hautes {\'E}tudes Sci. \textbf{117} (2013), 179--245.

\bibitem[BE21]{BE21} 
  S.~Boucksom and D.~Eriksson.
\newblock \emph{Spaces of norms, determinant of cohomology and Fekete points in 
  non-Archimedean geometry}.
\newblock Adv. Math. \textbf{378} (2021), Paper No. 107501, 124 pp.

\bibitem[BFJ15]{nama}  
S.~Boucksom, C.~Favre and M.~Jonsson.
\newblock \emph{Solution to a non-Archimedean Monge--Amp\`ere equation}.
\newblock J. Amer. Math. Soc. \textbf{28} (2015), 617--667.

\bibitem[BoJ17]{konsoib}
S.~Boucksom and M.~Jonsson. 
\newblock \emph{Tropical and non-Archimedean limits of degenerating families of volume forms}.
\newblock J. {\'E}c. Polytech. Math. \textbf{4} (2017), 87--139.

\bibitem[BoJ18]{trivval}
S.~Boucksom and M.~Jonsson. 
\newblock \emph{Singular semipositive metrics on line bundles on varieties over trivially valued fields}.
\newblock \texttt{arXiv:1801.08229v1}.

\bibitem[BGJ+20]{BGJKM20}
J.~I. Burgos Gil, W.~Gubler, P.~Jell, K.~K\"unnemann and F.~Martin
\newblock\emph{Differentiability of non-archimedean volumes and non-archimedean Monge--Amp\`ere equations (with an appendix by Robert Lazarsfeld)}.
\newblock Algebr. Geom. \textbf{7} (2020), 113--152.

\bibitem[BJKM21]{BJKM21}
J.~I. Burgos Gil, W.~Gubler, P.~Jell, K.~K\"unnemann and F.~Martin
\newblock\emph{Pluripotential theory for tropical toric varieties and non-archimedean Monge--Amp\`ere equations}.
\newblock \texttt{arXiv:2102.07392}.

\bibitem[BPS11]{BPS} 
J.~I.~Burgos Gil, P.~Philippon and M.~Sombra.
\newblock \emph{Arithmetic geometry of toric varieties. Metrics, measures, and heights}.
\newblock  Ast\'erisque No. 360 (2014).

\bibitem[Caf92]{Caf} 
L.~A.~Caffarelli.
\newblock \emph{The regularity of mappings with a convex potential}.
\newblock  J. Amer. Math. Soc.  {\bf 5} (1992), 99--104.

\bibitem[CL06]{CL06} 
A.~Chambert-Loir.
\newblock \emph{Mesures et {\'e}quidistribution
 sur les espaces de Berkovich}.
\newblock  J. Reine Angew. Math. \textbf{595} (2006), 215--235.

\bibitem[CLD12]{CLD} 
 A.~Chambert-Loir and A.~Ducros.
\newblock \emph{Formes diff{\'e}rentielles r{\'e}elles et courants 
  sur les espaces de Berkovich}.
\newblock \texttt{arXiv:1204.6277}.

\bibitem[CY82]{ChengYau}
S.~Y.~Cheng and S.-T. Yau.
\newblock \emph{The real Monge--Amp\`ere equation and affine flat structures}. \newblock Proceedings of the 1980 Beijing Symposium on Differential Geometry and Differential Equations, Vol. 1, 2, 3 (Beijing, 1980), 339--370.
\newblock Sci. Press Beijing, Beijing, 1982.

\bibitem[CJL21]{CJL}
T. Collins, A. Jacob, and Y.-S. Lin. \emph{Special Lagrangian tori in log Calabi--Yau manifolds}, Duke Math. J. 170 (7) 1291– 1375, May 15, 2021.

\bibitem[CT14]{CT14}
T.~Collins and V.~Tosatti.
\newblock \emph{An extension theorem for K\"ahler currents with analytic singularities}.
\newblock Ann. Fac. Sci. Toulouse Math. (6) {\bf 23} (2014), no. 4, 893--905.

\bibitem[CGZ13]{CGZ13}
D.~Coman, V.~Guedj and A.~Zeriahi. 
\newblock \emph{Extension of plurisubharmonic functions with growth control}.
\newblock J. Reine Angew. Math. {\bf 676} (2013), 33--49. 

\bibitem[CGZ22]{CGZ22}
D.~Coman, V.~Guedj and A.~Zeriahi. 
\newblock \emph{On the extension of quasiplurisubharmonic functions}.
\newblock Analysis Math. \textbf{48} (2022). 411--426.

\bibitem[CLS11]{CLS} 
D. A.~Cox,  J. B.~Little and H. K.~Schenck.
\newblock \emph{Toric varieties}.
\newblock American Mathematical Society, Providence, RI.
\newblock \textbf{124} (2011).

\bibitem[Del89]{Del89}
P.~Delano\"e.
\newblock\emph{Remarques sur les variet\'es localement Hessiennes}.
\newblock Osaka J. Math. \textbf{26} (1989), 65--69.

\bibitem[Fig99]{Figalli}
A.~Figalli.
\newblock \emph{The Monge--Amp\`ere equation and its applications}.
\newblock Z\"urich lectures in advanced mathematics.
\newblock European Mathematical Society (EMS), Z\"urich, 2017.

\bibitem[Fol99]{Folland}
G.~B.~Folland.
\newblock \emph{Real analysis: modern techniques and their
  applications, second edition}.
\newblock Pure and applied mathematics (New York). 
A Wiley-Interscience Publication.
\newblock John Wiley \& Sons, Inc., New York, 1999.

\bibitem[Ful93]{FultonToric}
W.~Fulton.
\newblock \emph{Introduction to toric varieties}.
\newblock Annals of Mathematics Studies, 131. 
Princeton University Press, Princeton, NJ, 1993.

\bibitem[Got22]{Got22}
K.~Goto.
\newblock\emph{On the two types of affine structures for degenerating Kummer surfaces: non-Archimedean vs Gromov--Hausdorff limits}.
\newblock \texttt{arXiv:2203.14543}.

\bibitem[GO22]{GO22}
K.~Goto and Y.~Odaka.
\newblock\emph{Special Lagrangian fibrations, Berkovich retraction, and crystallographic groups}.
\newblock \texttt{arXiv:2206.14474}.

\bibitem[Gro13]{Gro13}
M.~Gross.
\newblock\emph{Mirror symmetry and the Strominger--Yau--Zaslow conjecture}.
\newblock Current developments in mathematics 2012, 133--191.
\newblock Int. Press, Somerville, MA, 2013.

\bibitem[GS06]{GS06}
M.~Gross and B.~Siebert.
\newblock\emph{Mirror symmetry via logarithmic degeneration data I}.
\newblock J. Diff. Geom. \textbf{72} (2006), 169--338.

\bibitem[GW00]{GW00}
M.~Gross and B.~Wilson.
\newblock\emph{Large complex structure limits of K3 surfaces}.
\newblock J. Diff. Geom. \textbf{55} (2000), 475--546.

\bibitem[GTZ13]{GTZ13}
M.~Gross, V.~Tosatti and Y.~Zhang.
\newblock\emph{Collapsing of abelian fibered Calabi--Yau manifolds}.
\newblock Duke Math. J. \textbf{162} (2013), 517--551.

\bibitem[GTZ16]{GTZ16}
M.~Gross, V.~Tosatti and Y.~Zhang.
\newblock\emph{Gromov--Hausdorff collapsing of Calabi-Yau manifolds}.
\newblock Comm. Anal. Geom. \textbf{24} (2016), 93--113.

%\bibitem[TZ16]{TZ16}
%V.~Tosatti and Y.~Zhang.
%\newblock\emph{Collapsing hyperk\"ahler manifolds}. %(2016), 93--113.

\bibitem[Gub98]{GublerLocal}
W.~Gubler.
\newblock \emph{Local heights of subvarieties over non-Archimedean fields}.
\newblock J. Reine Angew. Math. \textbf{498} (1998), 61--113.

\bibitem[Gub07]{Gub07}
W.~Gubler.
\newblock \emph{Tropical varieties for non-archimedean analytic spaces}.
\newblock Invent. Math. \textbf{169} (2007), 321--376.

\bibitem[GT21]{GuedjTo}
V.~Guedj and T.~D.~T\^o.
\newblock \emph{Monge--Amp\`ere equations on compact Hessian manifolds}.
\newblock \texttt{arXiv:2106.14740}.

\bibitem[Hul]{HultOT}
J.~Hultgren.
\newblock \emph{Singular affine structures and Monge--Amp\`ere equations on some tropical manifolds}.
\newblock In preparation.

\bibitem[HO19]{HO19} 
J.~Hultgren and M.~\"Onnheim.
\newblock \emph{An optimal transport approach to Monge--Amp\`ere equations on compact Hessian manifolds}.
\newblock J. Geom. Anal. \textbf{29} (2019), 1953--1990.

\bibitem[H\"or94]{Hormander}
L.~H\"ormander.
\newblock \emph{Notions of convexity}.
\newblock Progress in Mathematics, 127.
\newblock Birkh\"auser Boston, Inc., Boston, MA, 1994.
%viii+414 pp. ISBN: 0-8176-3799-0

\bibitem[Kol13]{Kol13} 
J.~Koll{\'a}r.
\newblock \emph{Singularities of the minimal model program}.
\newblock Cambridge tracts in Mathematics, 200 (2013).
\newblock Cambridge University Press, Cambridge.

\bibitem[KS06]{KS06}
M.~Kontsevich and Y.~Soibelman.
\newblock\emph{Affine structures and non-Archimedean analytic spaces}.
\newblock In \emph{The unity of mathematics}.
\newblock Progr. Math. 244, 321--385.
\newblock  Birkh\"auser Boston, Boston, MA, 2006.

%\bibitem[Li20a]{LiUnifSkoda}
%Y.~Li.
%\newblock \emph{Uniform Skoda integrability and Calabi--Yau degeneration}.
%\newblock \texttt{arXiv:2006.16961}.
  
\bibitem[Li20]{LiSYZ}
Y.~Li.
\newblock \emph{Metric SYZ conjecture and non-archimedean geometry}.
\newblock \texttt{arXiv:2007.01384}.

\bibitem[Li22a]{LiFermat}
Y.~Li.
\newblock \emph{Strominger--Yau--Zaslow conjecture for Calabi--Yau hypersurfaces in the Fermat family}.
\newblock Acta Math. \textbf{229} (2022), 1--53.

\bibitem[Li22b]{LiSurvey}
Y.~Li.
\newblock \emph{Survey on the metric SYZ conjecture and non-archimedean geometry}.
\newblock \texttt{arXiv:2204.11363}.

\bibitem[Li23]{LiFano}
Y.~Li.
\newblock \emph{Metric SYZ conjecture for certain toric Fano hypersurfaces}.
\newblock \texttt{arXiv:2301.12983}.

\bibitem[LT20]{LiTosatti}
Y.~Li. and V.~Tosatti
\newblock\emph{Diameter bounds for degenerating Calabi-Yau metrics}.
\newblock \texttt{arXiv:2006.13068}, to appear in J. Differential Geom.

\bibitem[MS15]{MS15}
D.~Maclagan and B.~Sturmfels.
\newblock \emph{Introduction to tropical geometry}.
\newblock Graduate Studies in Mathematics.
\newblock American Mathematical Society, Providence, RI.
\newblock \textbf{161} (2015).

\bibitem[MMRZ]{MMRZ22}
C.~Y. Mak, D.~Matessi, H.~Ruddat and I.~Zharkov.
\newblock\emph{Lagrangian Strominger--Yau--Zaslow fibrations}.
\newblock In preparation.

\bibitem[Moo15]{Moo15}
C.~Mooney.
\newblock\emph{Partial regularity for singular solutions to the Monge--Amp\`ere equation}.
\newblock Comm. Pure Appl. Math. 68 (2015), 1066--1084.

\bibitem[Moo21]{Moo21}
C.~Mooney.
\newblock\emph{Solutions to the Monge--Amp{\`e}re equation with polyhedral and Y-shaped singularities}.
\newblock J. Geom. Anal. \textbf{31} (2021), 9509--9526.

\bibitem[MR22]{MR22}
C.~Mooney and A.~Rakshit.
\newblock\emph{Solutions structures in solutions to the Monge--Amp\`ere equations with point masses}.
\newblock \texttt{arXiv:2204.11365}.

\bibitem[MN15]{MN15}
M.~Musta\c{t}\v{a} and J.~Nicaise. 
\newblock\emph{Weight functions on non-archimedean analytic spaces and
the Kontsevich--Soibelman skeleton}.
\newblock Algebr. Geom. \textbf{2} (2015), 365--404. 

\bibitem[MP21]{MP21}
E.~Mazzon and L.~Pille-Schneider.
\emph{Toric geometry and integral affine structures in non-archimedean mirror symmetry}.
\newblock \texttt{arXiv:2110.04223}.

\bibitem[NX16]{NX16}
J.~Nicaise and C.~Xu.
\newblock\emph{The essential skeleton of a degeneration of algebraic varieties}.
\newblock American Journal of Mathematics \textbf{138}  (2016), 1645--1667.

\bibitem[NXY19]{NXY19}
J.~Nicaise, C.~Xu and T.~Y.~Yu.
\newblock\emph{The non-Archimedean SYZ fibration}.
\newblock Compositio Math. \textbf{155}  (2019), 953--972.

\bibitem[NWZ21]{NWZ21}
J.~Ning,  Z.~Wang and X.~Zhou
\newblock \emph{On the extension of K\"ahler currents on compact K\"ahler manifolds: holomorphic retraction case}.
\newblock \texttt{arXiv:2105.08224}.

\bibitem[Oda20]{Oda20}
Y.~Odaka.  
\newblock\emph{Degenerated Calabi-Yau varieties with infinite components, Moduli compactifications, and limit toroidal structures}.
\newblock \texttt{arXiv:2011.12748}.

\bibitem[OO21]{OO21}
Y.~Odaka and Y.~Oshima.  
\newblock\emph{Collapsing K3 surfaces, tropical geometry and moduli compactifications of Satake, Morgan--Shalen type}.
\newblock MSJ Memoirs, 40.
\newblock Mathematical Society of Japan, Tokyo, 2021. 

\bibitem[PS22]{PS22} 
L. Pille-Schneider.
\newblock\emph{Hybrid toric varieties and the non-archimedean SYZ fibration on Calabi--Yau hypersurfaces}.
\newblock \texttt{arXiv:2210.05578}.

\bibitem[RS20]{RS20} 
H.~Ruddat and B.~Siebert.
\newblock\emph{Period integrals from wall structures via tropical cycles, canonical coordinates in mirror symmetry and analyticity of toric degenerations}.
\newblock  Publ. Math. Inst. Hautes \'Etudes Sci. \textbf{132} (2020), 1--82.

\bibitem[RZ]{RZ22} 
H.~Ruddat and I.~Zharkov.
\newblock\emph{Topological Strominger-Yau-Zaslow fibrations}.
\newblock In preparation.

\bibitem[SYZ96]{SYZ} 
A.~Strominger, S.-T. Yau, E. Zaslow. 
\newblock \emph{Mirror Symmetry is T-duality}. 
\newblock Nucl. Phys. B479:243--259, 1996.

\bibitem[Vil21]{Vilsmeier}
C.~Vilsmeier.
\newblock \emph{A comparison of the real and non-archimedean Monge--Amp\`ere operator}.
\newblock Math. Z. \textbf{297} (2021),  633--668.

\bibitem[WZ20]{WZ20}
Z.~Wang and X.~Zhou.
\newblock \emph{On the extension of K\"ahler currents on compact K\"ahler manifolds}.
\newblock \texttt{arXiv:2002.11381}.

\bibitem[Yau78]{Yau}
S.~T.~Yau.
\newblock\emph{On the Ricci curvature of a compact K\"ahler 
  manifold and the complex Monge--Amp{\`e}re equation}.
\newblock  Comm. Pure Appl. Math.  \textbf{31}  (1978), 339--411. 

\bibitem[YZ17]{YZ17}
X.~Yuan and S.-W.~Zhang.
\newblock \emph{The arithmetic Hodge index theorem for adelic line bundles}.
\newblock Math. Ann. \textbf{367} (2017), 1123--1171.

\bibitem[Zha95]{Zha95}
S.-W.~Zhang.
\newblock\emph{Positive line bundles on arithmetic varieties}.
\newblock  J. Amer. Math. Soc. \textbf{8} (1995), 187--221.

\end{thebibliography}
\end{document}